\documentclass[11pt]{article}

\usepackage{amsmath,amsthm,amssymb,mathrsfs}
\usepackage{graphicx,mdframed}

\font\tencmmib=cmmib10 \skewchar\tencmmib '60
\newfam\cmmibfam
\textfont\cmmibfam=\tencmmib

\def\lessim{\ \lower4pt\hbox{$
		\buildrel{\displaystyle <}\over\sim$}\ }
\def\gessim{\ \lower4pt\hbox{$\buildrel{\displaystyle >}
		\over\sim$}\ }

\newcommand{\la}{\langle}
\newcommand{\ra}{\rangle}

\newcommand{\e}{\mathbb{E}}
\newcommand{\p}{\mathbb{P}}

\newcommand{\ME}{\mbox{\it ME}}
\newcommand{\MCE}{\mbox{\it MCE}}
\newcommand{\CF}{\mbox{\it CF}}

\newtheorem{lemma}{\bf Lemma}

\newtheorem{theorem}{\bf Theorem}
\newtheorem{corollary}{\bf Corollary}
\newtheorem{remark}{\bf Remark}

\newtheorem{proposition}{\bf Proposition}

\newenvironment{Proof of lemma}{\noindent{\bf Proof of Lemma}}{\hfill$\Box$\newline}
\newenvironment{Proof of theorem}{\noindent{\it Proof of Theorem}}{\hfill\scriptsize{$\Box$}\newline}
\newenvironment{Proof of theorems}{\noindent{\bf Proof of Theorems}}{\hfill$\Box$\newline}
\newenvironment{Proof of proposition}{\noindent{\bf Proof of Proposition}}{\hfill$\Box$\newline}
\newenvironment{Proof of propositions}{\noindent{\bf Proof of Propositions}}{\hfill$\Box$\newline}
\newenvironment{Proof of exercise}{\noindent{\it Proof of Exercise:}}{\hfill$\Box$}
\title{On the energy landscape of spherical spin glasses}

\author{Antonio Auffinger \thanks{Department of Mathematics, Northwestern University, tuca@northwestern.edu, research partially supported by NSF Grant CAREER DMS-1653552 and NSF Grant DMS-1517894.} \\
	\small{Northwestern University}\and Wei-Kuo Chen \thanks{School of Mathematics, University of Minnesota. Email: wkchen@umn.edu, research partially supported by NSF grant DMS-1642207 and Hong Kong Research Grants Council GRF-14302515.}\\
	\small{University of Minnesota}}

\begin{document}
	\maketitle
	\begin{abstract} We investigate the energy landscape of the spherical mixed even $p$-spin model near its maximum energy. We relate the distance between pairs of near maxima to the support of the Parisi measure at zero temperature. We then provide an algebraic relation that characterizes one-step replica symmetric breaking Parisi measures. For these measures, we show that any two nonparallel spin configurations around the maximum energy are asymptotically orthogonal to each other. In sharp contrast, we study models with full replica symmetry breaking and show that all possible values of the asymptotic distance are attained near the maximum energy.  
		
	\end{abstract}

	
	\section{Introduction and main results}
	This work deals with  geometric properties of general Gaussian smooth functions on the $N$ dimensional sphere as $N$ goes to infinity. The questions addressed in this paper can be phrased as: Where are the peaks of a random Morse function in a high dimensional sphere? How can we travel between two peaks and what is their typical spherical distance? 
	
	A rich description of the landscape of these functions is predicted by the theory of mean-field spin glasses. The functions that we consider here are known as the Hamiltonians of mixed spherical $p$-spin models.  Our main result relates the above questions to the structure of the Parisi measure of these models at zero temperature. We confirm and make precise a common prediction by physicists, that the landscape of these functions near the maxima heavily depends on the number of levels of replica symmetry breaking (RSB). For  references in the physics literature the reader is invited to see \cite{MPV}. For applications of spin glass theory in computer science, neural networks and more see \cite{MeMo} and the references therein.

	We now describe the functions that we analyze in the terminology of spin glass theory.
	Let $S_N$ be the sphere $$\left \{\sigma\in\mathbb{R}^N:\sum_{i=1}^N\sigma_i^2=N \right \}.$$ Consider the Hamiltonian of the spherical mixed even $p$-spin model indexed by $S_N,$
	\begin{align}\label{ham}
	H_N(\sigma)&=X_N(\sigma)+h\sum_{i=1}^N\sigma_i
	\end{align}
	for 
	$$
	X_N(\sigma):=\sum_{p\in 2\mathbb{N}}\frac{c_p^{1/2}}{N^{(p-1)/2}}\sum_{1\leq i_1,\ldots,i_p\leq N}g_{i_1,\ldots,i_p}\sigma_{i_1}\cdots\sigma_{i_p}, 
	$$
	where $g_{i_1,\ldots,i_p}$'s are i.i.d. standard Gaussian, $h\geq 0$ denotes the strength of an external field, and the sequence $(c_p)_{p\in 2\mathbb{N}}$ satisfies $c_p\geq 0$, $\sum_{p\in 2\mathbb{N}}c_p=1,$ and 
	\begin{equation}\label{eq:cond}
	\sum_{p\in 2\mathbb{N}}2^pc_p <\infty.
	\end{equation}
	It is easy to check that
	\begin{align*}
	\e X_N(\sigma^1)X_{N}(\sigma^2)=N\xi(R_{1,2}),
	\end{align*}
	where $$R_{1,2}:=\frac{1}{N}\sum_{i=1}^N\sigma_i^1\sigma_i^2$$ is the normalized inner-product between $\sigma^{1}$ and $\sigma^{2}$, known as the overlap, and $$\xi(s):=\sum_{p\in 2\mathbb{N}}c_ps^p.$$ 
	Condition \eqref{eq:cond} is more than enough to guarantee that the sum \eqref{ham} is almost surely finite, and the energy $H_{N}$ is a.s. smooth and Morse; see Theorem 11.3.1 of \cite{AT}. The simplest case is the spherical Sherrington-Kirkpatrick (SK) model, $\xi(s)=s^2$.

	We are interested in the collection of points $\sigma \in S_{N}$ such that $H_{N}(\sigma)$ is close to the maximum value of $H_{N}$. For this, denote 
	the maximum energy (ME) of $H_{N}$ by
	\begin{align*}
	\ME_{N} &=\max_{\sigma\in S_N}\frac{H_N(\sigma)}{N}.
	\end{align*}
	Recently, Chen-Sen \cite{ArnabChen15} and Jagannath-Tobasco \cite{JT} showed that the limiting maximum energy can be computed through a variational principle, similar to the Crisanti-Sommers formula \cite{CS}. More precisely, let $\mathcal{K}$ be the collection of all measures $\nu$ on $[0,1]$, which takes the form,
	$$
	\nu(ds)=1_{[0,1)}(s)\gamma(s)ds+\Delta \delta_{\{1\}}(ds),
	$$
	where $\gamma(s)$ is a nonnegative and nondecreasing function on $[0,1)$ with right-continuity, $\Delta>0,$ and $\delta_{\{1\}}$ is a Dirac measure at $1.$ Define the Crisanti-Sommers functional by
	\begin{align*}
	\mathcal{Q} (\nu)&=\frac{1}{2}\Bigl(\int_0^1(\xi'(s)+h^2)\nu(ds)+\int_0^1\frac{dq}{\nu((q,1])}\Bigr)
	\end{align*}
	for $\nu\in \mathcal{K}.$ The Crisanti-Sommers formula for the maximum energy derived in Chen-Sen \cite{ArnabChen15}\footnote{Although the form in Chen-Sen \cite{ArnabChen15} is not exactly the same as \eqref{CS}, it can be easily expressed in terms of the current form \eqref{CS} by performing a change of variable, $\Delta=L-\int_0^1\gamma(s)ds.$} and Jagannath-Tobasco \cite{JT} states that 
	\begin{align}\label{CS}
	\ME :=\lim_{N\rightarrow\infty}\ME_{N} &=\inf_{\nu\in\mathcal{K}}\mathcal{Q} (\nu).
	\end{align}
	Note that $\mathcal{Q}$ is a strictly convex functional on $\mathcal{K}$ and it was proved in \cite{ArnabChen15,JT} that the right-hand side has a unique minimizer, denoted by $$\nu_P(ds)=\gamma_P(s)1_{[0,1)}(s)ds+\Delta_P\delta_{\{1\}}(ds).$$
	We denote by $\rho_P$ the measure on $[0,1)$ induced by $\gamma_P$, i.e., 
	\begin{align}
	\label{rho}
	\gamma_P(s)=\rho_P([0,s]),\,\,\forall s\in[0,1).
	\end{align}
	We call $\rho_P$ the Parisi measure at zero temperature.


	\subsection{Two general principles}
	
		For fixed $\eta>0,$ our main theorems relate the geometry of the set of spin configurations near the maximum energy
		\begin{equation}\label{eq:mainset}
		\mathcal L(\eta):=\big \{ \sigma \in S_{N}: H_{N}(\sigma) > N (\ME - \eta)\big\}.
		\end{equation}
		to the structure of the Parisi measure $\rho_P.$ Clearly $\mathcal{L}(\eta_1)\subseteq\mathcal{L}(\eta_2)\subseteq S_N$ for $0<\eta_1<\eta_2.$
		
	\subsubsection{Relevance of the Parisi measure}
     For  fixed $\eta >0$ and Borel measurable set $A\subset [-1,1],$ set
	\begin{align*}
	\mathbb{P}_N(\eta,A)&:=\p\bigl(\exists \; \mbox{$\sigma^1,\sigma^2\in \mathcal L(\eta)$ with $R_{1,2}\in A$}\bigr).
	\end{align*}	
	In other words, $\mathbb{P}_N(\eta,A)$ is the probability that there exist two spin configurations near the maximum energy and their overlap lies in $A.$	
	Denote by 
	\begin{align}\begin{split}
	\label{GP:eq1}
	\Gamma&=(\mbox{supp}\rho_P)\cup\{1\},\\
	s_P&=\min \Gamma.
	\end{split}
	\end{align}  
	The following proposition summarizes some properties of $s_P:$
	
	\begin{proposition}
		\label{prop0}
		The quantity $s_P$ obeys the following statements:
		\begin{itemize}
			\item[$(i)$] If $h=0,$ then $s_P=1$ when $\xi(s)=s^2$ and $s_P=0$ when $c_p\neq 0$ for at least one even $p\geq 4$.
			\item[$(ii)$] If $h\neq 0$, then $s_P>0.$ 
		\end{itemize}
	\end{proposition}
 
	Note that since $X_N$ involves only even spin interactions, when the external field vanishes, $H_N$ is symmetric, i.e., $H_N(-\sigma)=H_N(\sigma).$ 
	Our first main result states that in the absence of external field, for any given $u\in[-1,1]$ with $|u|\in \Gamma$, with overwhelming probability there exist two spin configurations around the maximum energy such that their overlap is around $u.$

	\begin{theorem}\label{thm1} Assume $h=0.$ Let $u\in[-1,1]$ with $|u|\in\Gamma$. For any $\varepsilon,\eta>0$, there exists $K>0$ such that for all $N\geq 1,$
				\begin{align}\label{thm1:eq1}
				\mathbb{P}_N\bigl(\eta,(u-\varepsilon,u+\varepsilon)\bigr)&\geq 1-Ke^{-\frac{N}{K}}.
				\end{align}
	\end{theorem}
		
	In the case that the external field is present, i.e., $h\neq 0,$ the Hamiltonian is no longer symmetric and Proposition \ref{prop0}$(ii)$ asserts $s_P>0.$ An analogous result of Theorem \ref{thm1} remains valid. Furthermore, the overlap between any two spin configurations near the maximum energy does not lie in $[-1,s_P).$
	
	\begin{theorem}\label{thm2} Assume $h\neq 0.$ 
		\begin{itemize}
			\item[$(i)$] Let $u\in \Gamma $. 
			For any $\varepsilon,\eta>0,$ there exists $K$ independent of $N$ such that for all $N\geq 1,$
			\begin{align}\label{thm2:eq1}
			\mathbb{P}_N\bigl(\eta,(u-\varepsilon,u+\varepsilon)\bigr)&\geq 1-Ke^{-\frac{N}{K}}.
			\end{align}
			\item[$(ii)$] For any $\varepsilon>0,$ there exist $\eta,K>0$ such that for all $N\geq 1,$
			\begin{align}
			\label{thm2:eq2}
			\mathbb{P}_N\bigl(\eta,[-1,s_P-\varepsilon]\bigr)&\leq Ke^{-\frac{N}{K}}.
			\end{align}
		\end{itemize}
	\end{theorem}
	
	
	In view of Theorems \ref{thm1} and \ref{thm2}, one would wonder what the corresponding result could be when the overlap is restricted to $[s_P,1]\setminus \Gamma.$ In Section \ref{sec2}, we explore three cases of the mixed even $p$-spin model, where we show that the probability of having two spin configurations near the maximum energy with overlap inside $[s_P,1]\setminus \Gamma$ is exponentially small. 
	
	\subsubsection{An equidistant structure}
	
    For any fixed $q\in \Gamma$, Theorems \ref{thm1} and \ref{thm2} assert the existence of a pair of spin configurations near the maximum energy with overlap around $u.$ Our second principle here shows that if we take $q=0$ when $h=0$ and $q=s_P$ when $h\neq 0$, then there exist exponentially many equidistant spin configurations near the maximum energy. For any $\varepsilon,\eta,K>0$ and $q\in[0,1],$ denote by 
    \begin{align*}
    \p_N(\varepsilon,\eta,q,K)
    \end{align*} 
    the probability that there exists a subset $O_N\subset S_N$ such that
    \begin{enumerate}
    	\item[$(i)$] $O_N\subset \mathcal{L}(\eta).$
    	\item[$(ii)$] $O_N$ contains at least $Ke^{N/K}$ many elements.
    	\item[$(iii)$] $|R(\sigma,\sigma')-q|\leq\varepsilon$ for all distinct $\sigma,\sigma'\in O_N.$
    \end{enumerate}
    Denote by $q_{0}=0$ if $h=0$ and $q_{0}=s_P$ if $h\neq 0.$  Our main result is stated as follows.
   
	\begin{proposition}
		\label{thm1.1}
		For any $\varepsilon,\eta>0$, there exists $K>0$ such that for any $N\geq 1,$
		\begin{align*}
		\p_N(\varepsilon,\eta,q_0,K)\geq 1-Ke^{-N/K}.
		\end{align*}
	\end{proposition}
	
	A major feature of Proposition \ref{thm1.1} is that when the external field vanishes $h=0,$ we can always find exponentially many orthogonal spin configurations around the maximum energy for any mixture $\xi$. One may find an analogous statement of Proposition \ref{thm1.1} in \cite{CHL16} in the setting of the mixed even $p$-spin model with Ising spin configuration space.

	\subsubsection{Ideas of the proof}
	Before moving to our examples, we briefly sketch the main approach and perspective of this paper and compare to the existing results. Our approach to Theorems \ref{thm1} and \ref{thm2} is via the  maximum of the coupled energy (MCE) with overlap constraint,
	\begin{align*}
	\MCE_N(A)&:=\frac{1}{N}\e \max_{R_{1,2}\in A}\bigl(H_N(\sigma^1)+H_N(\sigma^2)\bigr).
	\end{align*} 
	Here, $A$ is a Borel measurable subset of $[-1,1].$ In particular, we care for which sets $A$, $\MCE_N(A)$ and $2\ME_N $ are asymptotically the same. When this occurs, it means that one can always find two spin configurations, whose energies are around the global maximum and the overlap is in $A.$ If $\MCE_N(A)$ and $2\ME_N$ are asymptotically different, then the overlap between any two spin configurations around the maximum energy does not lie in $A.$  While it is in general very difficult to compare the values of two extrema Gaussian fields, it turns out that the current case is achievable and the set $A$ depends closely on the Parisi measure $\rho_P.$
	
	The above strategy is different from the approaches used in the recent studies of the landscape of spherical $p$-spin models, especially those connected to the complexity of such functions \cite{ABC, AB, S1, S2, SZ}. Here, we neither rely on the use of the Kac-Rice formula, nor restrict ourselves to the study of local maxima or critical points. Of course, inside each connected component of $\mathcal L(\eta)$ there exists at least one local maxima of $H_{N}$. As we will see in the next section, this fact combined with Theorem \ref{thm:1RSB:os} below provides a different proof and extends the results of Subag \cite{S2} about the orthogonality of critical points in the pure $p$-spin model (See Remark \ref{SubagRem}). Another advantage of our approach is that it also allows to establish Theorems \ref{thm1} and \ref{thm2} in the setting of the mixed even $p$-spin models with Ising-spin configuration space following an identical argument.

	\subsection{Levels of replica symmetry breaking at zero temperature}\label{sec2}

	In this section, we  explore the consequences of Theorems \ref{thm1} and \ref{thm2} depending on the structure of the support of the Parisi measure $\rho_{P}$.	
	We say that $\rho_P$ is replica symmetric (RS) if $\rho_P\equiv 0$ on $[0,1)$, has $k$-step replica symmetry breaking ($k$RSB) if $\rho_P=\sum_{i=1}^kA_i\delta_{\{q_i\}}$ for some $A_1,\ldots,A_k>0$ and distinct $q_1,\ldots,q_k\in[0,1),$ and has full replica symmetry breaking (FRSB) otherwise. Under different conditions on $\xi$ and $h$, examples devoted to RS, 1RSB, and FRSB were discussed in Chen-Sen \cite{ArnabChen15}, while Jagannath-Tobasco \cite{JT} presented a description on the structure of the Parisi measure in general situations.

	\subsubsection{RS solution}
	
	In the first example, we consider the mixed even $p$-spin model, whose $\xi$ and $h$ satisfy \begin{align}
	\label{thm:RS1}
	\xi''(1)\leq\xi'(1)+h^2.
	\end{align}
	In \cite[Proposition 1]{ArnabChen15}, it was shown that this is a sufficient and necessary condition to guarantee that the Parisi measure of the Crisanti-Sommers formula \eqref{CS} is  replica symmetric. In this case, it was readily computed in \cite[Proposition 1]{ArnabChen15} that 
	\begin{align}\begin{split}\label{RS:eq1}
	\nu_P (ds)&=\bigl(\xi'(1)+h^2\bigr)^{-1/2}\delta_{\{1\}}(ds),\\
	\ME &=\bigl(\xi'(1)+h^2\bigr)^{1/2}.
	\end{split}
	\end{align}
	Therefore, $\Gamma=\{1\}.$
	
	\begin{theorem}[RS]\label{thm:RS}
		Assume $h\neq 0$ and \eqref{thm:RS1} holds.		
		For any $\varepsilon\in (0,1),$ there exist $\eta,K>0$ such that
		\begin{align*}
		\mathbb{P}_N\bigl(\eta,[-1,1-\varepsilon]\bigr)&\leq Ke^{-\frac{N}{K}}.
		\end{align*}
	\end{theorem}
	
	This theorem says that if the strength of the external field $h$ dominates the mixed $p$-spin interactions $X_N$, i.e., \eqref{thm:RS1} holds, then any two spin configurations with energies near the global maximum must be very close to each other. The picture of Theorem \ref{thm:RS} will change drastically if one considers different mixtures. 
	
	\subsubsection{FRSB solution}
	
    The second example is the mixed even $p$-spin model with FRSB Parisi measure. Assume that the external field $h$ no longer dominates $X_N$, i.e., $\xi''(1)>\xi'(1)+h^2$. Suppose that $1/\sqrt{\xi''}$ is concave on $(0,1]$. 
	Recall $\Gamma$ from \eqref{GP:eq1}. From \cite[Proposition 2]{ArnabChen15}, it was computed that 
	\begin{align*}
	\nu_P (ds)&=\gamma_P (s)1_{[0,1)}(s)ds+\xi''(1)^{-1/2}\delta_{\{1\}}(ds),\\
	\ME &=s_P\xi''(s_P)^{1/2}+\int_{s_P}^1\xi''(s)^{1/2}ds,
	\end{align*}
	where $s_P\in[0,1]$ is the unique solution to 
	\begin{align*}
	s_P\xi''(s_P)=\xi'(s_P)+h^2
	\end{align*}
	and
	\begin{align}\begin{split}
	\label{eq2}
	\gamma_P (s)&=\left\{
	\begin{array}{ll}
	0,&\mbox{if $s\in[0,s_P)$},\\
	\frac{\xi'''(s)}{2\xi''(s)^{3/2}},&\mbox{if $q\in [s_P,1)$}.
	\end{array}		\right.
	\end{split}
	\end{align}
	From \eqref{eq2}, the Parisi measure $\rho_P$ is supported on $[s_P,1)$ and thus it is FRSB. Our results below present a completely different behavior compared to Theorem \ref{thm:RS} if one considers the  opposite region of \eqref{thm:RS1}.

	\begin{theorem}[FRSB]\label{FRSB}
		Assume $\xi''(1)>\xi'(1)+h^2$ and $1/\sqrt{\xi''}$ is concave on $(0,1]$. 
		\begin{itemize}
			\item[$(i)$] Assume $h=0.$ Let $u\in [-1,1].$ For any $\varepsilon,\eta>0,$ there exist $K>0$ such that
			\begin{align*}
			\mathbb{P}_N\bigl(\eta,(u-\varepsilon,u+\varepsilon)\bigr)&\geq 1-Ke^{-\frac{N}{K}}.
			\end{align*}
			\item[$(ii)$] Assume $h\neq 0.$ We have that
			\begin{itemize}
				\item[$(ii')$] Let $u\in [s_P,1].$ For any $\varepsilon,\eta>0,$ there exists $K>0$ such that
				\begin{align*}
				\mathbb{P}_N\bigl(\eta,(u-\varepsilon,u+\varepsilon)\bigr)&\geq 1-Ke^{-\frac{N}{K}}.
				\end{align*}
				\item[$(ii'')$] For any $\varepsilon>0$, there exists $\eta,K>0$ such that
				\begin{align*}
				\mathbb{P}_N\bigl(\eta,[-1,s_P-\varepsilon]\bigr)&\leq Ke^{-\frac{N}{K}}.
				\end{align*}
			\end{itemize}
		\end{itemize}
	\end{theorem}

		This theorem shows that for any $\eta >0$, the overlap attains any possible value of $[s_P,1]$ in the set $\mathcal L(\eta)$. As far as we know, this is the first rigorous result in spherical models that matches the physicists' expectation that, in models with FRSB, local maxima of the Hamiltonian $H_{N}$ slightly below the maximum energy should be separated by only $O(1)$ barriers. More precisely, the barrier between two local maxima $\sigma$ and $\sigma'$ is defined as
		$$
		B_N(\sigma,\sigma'):=\inf_{\tau}\max_{0\leq t\leq 1}\bigl(|H_N(\sigma)-H_N(\tau(t))|,|H_N(\sigma')-H_N(\tau(t))|\bigr),
		$$
		where the infimum is taken over all continuous paths $\tau:[0,1]\mapsto S_N$ with $\tau(0)=\sigma$ and $\tau(1)=\sigma'.$ For FRSB models, it is expected that $B_N/N^c\rightarrow 0$ for any $c>0,$ see \cite[Section 9]{CK} for detail. This is in deep contrast with 1RSB models where local maxima near the maximum energy are orthogonal to each other with $B_N=O(N)$ barrier separating them, see Theorem \ref{thm:1RSB:os} below.

	\subsubsection{1RSB solution}
	
	Let $z\geq 0$ be the unique solution to
	\begin{align}\label{1RSB:eq0}
	\frac{1}{\xi'(1)}&=\frac{1+z}{z^2}\log (1+z)-\frac{1}{z}.
	\end{align}
	Note that the right hand-side is a strictly decreasing function and decreases from $1/2$ to 0 as $z$ tends from 0 to infinity. Since $$2=2\sum_{p\in 2\mathbb{N}}c_p\leq \sum_{p\in2\mathbb{N}}pc_p=\xi'(1),$$ the solution $z$ to \eqref{1RSB:eq0} is ensured. Also note that $z=0$ if and only if $\xi(s)=s^2,$ the spherical SK model. If $c_p\neq 0$ for at least one $p\geq 4$, then $z>0$ and we define \begin{align}
	\label{1RSB:lem-1:eq1}
	\zeta(s)&=\xi(s)+\xi'(s)(1-s)+\frac{\xi'(s)}{z}-\frac{(1+z)\xi'(1)}{z^2}\log\Bigl(1+\frac{z\xi'(s)}{\xi'(1)}\Bigr).
	\end{align}
	Here $\zeta(0)=\zeta(1)=0.$ For $h=0,$ our main result below gives a full characterization of the mixture parameter $\xi$ for the 1RSB Parisi measure with $\mbox{supp}\rho_P=\{0\}.$
	
	\begin{theorem}[1RSB]\label{thm:1RSB}
		Assume $h=0.$ The Parisi measure $\rho_P$ is 1RSB with $\mbox{supp}\rho_P=\{0\}$ if and only if $c_p>0$ for at least one $p\geq 4$ and 
		\begin{align}\label{thm:1RSB:eq1}
		\zeta(s)&\leq  0, \,\,\forall s\in[0,1].
		\end{align}
		In this case, 
		\begin{align}
		\begin{split}
		\label{thm1:1RSB:eq1}
		\nu_P(ds)&=\frac{z1_{[0,1)}(s)ds}{\sqrt{(1+z)\xi'(1)}}+\frac{\delta_{\{1\}}(ds)}{\sqrt{(1+z)\xi'(1)}},
		\end{split}\\
		\begin{split}
		\label{thm1:1RSB:eq2}
		\ME&=\frac{\xi'(1)+z}{\sqrt{(1+z)\xi'(1)}}.
		\end{split}
		\end{align}
	\end{theorem}
	
	If the inequality \eqref{thm:1RSB:eq1} is strict, we obtain a description of the energy landscape of the model around the maximum energy. 
	
	\begin{theorem}[Orthogonal structure]\label{thm:1RSB:os}
		Let $h=0$. If   \begin{align}\label{thm:1RSB:os:eq1}
		\zeta(s)&<0, \,\,\forall s\in(0,1),
		\end{align} 
		then for any $\varepsilon>0$, there exist $\eta,K>0$ such that for all $N\geq 1,$
		\begin{align}\label{1RSB:eq1}
		\mathbb{P}_N\bigl(\eta,[-1+\varepsilon,-\varepsilon]\cup[\varepsilon,1-\varepsilon]\bigr)&\leq Ke^{-\frac{N}{K}}.
		\end{align}
	\end{theorem}
	
	Theorem \ref{thm:1RSB:os} reads that with overwhelmingly probability, any two nonparallel spin configurations around the maximum energy are nearly orthogonal to each other. In other words, if one  wishes to travel between any two such spin configurations along a path on the energy landscape, then one unavoidably needs to climb down to a lower energy level at some point. Furthermore, recall the set $\mathcal{L}(\eta)$ from \eqref{eq:mainset}. Theorem \ref{thm:1RSB:os} combined with Theorem \ref{thm1.1} ($s_P=0$) implies that the number of nearly orthogonal components of $\mathcal{L}(\eta)$ is at least of exponential order.			
	
	The assumption \eqref{thm:1RSB:os:eq1} is numerically easy to check. Nonetheless, the following theorem provides a simple sufficient criterion for \eqref{thm:1RSB:os:eq1}.

	\begin{theorem}\label{cor1}
		Let $h=0.$ If 
		\begin{align}
		\label{cor:eq1}
		\xi'(1)>\xi''(0)(1+z)
		\end{align}
		and
		\begin{align}
		\label{cor:eq2}
		\frac{s}{\xi'(s)}\,\,\mbox{is convex on $(0,1)$},
		\end{align} then the strict inequality \eqref{thm:1RSB:os:eq1} holds.
	\end{theorem}
	
	Note that $\xi''(0)=c_2$. If $c_2=0$, then \eqref{cor:eq1} is redundant  and one only needs to verify \eqref{cor:eq2}. An important example of Theorem~\ref{cor1} is the pure $p$-spin model, i.e., $\xi(s)=s^{p}$ for $p\geq 4.$ In this case, the maximum energy and the Parisi measure were previously computed in \cite{ABC} and \cite[Proposition 3]{ArnabChen15}, which agree with \eqref{thm1:1RSB:eq1}. 
	
	\begin{remark}
		\rm The condition \eqref{cor:eq2} should be compared with the well-known criterion of testing 1RSB Parisi measure at both positive and zero temperatures  in \cite{CS,JT,Tal06}, where it was shown that the Parisi measure is either RS or 1RSB if $1/\sqrt{\xi''}$ is convex in $(0,1)$. Reportedly, there exists some $\xi$, which satisfies \eqref{cor:eq2}, but  $1/\sqrt{\xi''}$ is not convex on $(0,1)$. However, it is not clear to us whether the convexity of $1/\sqrt{\xi''}$ always implies that of $s/\xi'(s).$ 
	\end{remark}
	
	It is easy to construct models satisfying Theorem \ref{cor1}. The corollary below deals with mixture of two spin interactions. 		
	
	\begin{corollary}\label{cor2}
		Consider the spherical $(p+q)$-spin model with $h=0$ and $p,q\geq 4,$ i.e., $$
		\xi(s)=cs^p+(1-c)s^q$$
		for some $c\in[0,1].$
		If
		\begin{align}\label{cor2:eq2}
		2pq+4\geq 3(p+q)+(p-q)^2,
		\end{align}  
		then both conditions \eqref{cor:eq1} and \eqref{cor:eq2} are valid.
	\end{corollary}

	\begin{remark}\label{SubagRem}\rm Several authors studied the energy landscape of the $p$-spin model in recent years. The averaged complexity of critical points of $H_{N}$ was found in Auffinger-Ben Arous-\v{C}er{n}\'y \cite{ABC} and in Auffinger-Ben Arous \cite{AB}. Later, for the pure $p$-spin model, concentration of the complexity of the local maxima was established by Subag \cite{S1}. The energy landscape of the pure $p$-spin model around the maximum energy coincides with the picture described above. Theorem~\ref{thm:1RSB:os} works not only for the pure $p$-spin model, but also for any mixture such that $\zeta(s)>0$.  For an example of $\xi$ that involves infinitely many interactions in $X_N$, one could take $$
		\xi(s)=\frac{e^{s^2}-1-s^2}{e-2}.$$ 
		Thus, we recover and extend the orthogonality structure of local maxima discovered in \cite{S1} (see also \cite[Corollary 13]{S2}) to other models.
	\end{remark}
	
	\begin{remark} \rm It would be interesting to decide if condition $\zeta(s) <0$ coincides with the definition of pure-like models introduced in \cite{ABC} and also investigated in \cite{JT}.		
	\end{remark}

	\section{Parisi's formula and RSB bound for the free energies}
	
	In this section we review some well-known results from Talagrand \cite{Tal06} on the Parisi formula for the free energy and the Guerra-Talagrand RSB bound for the coupled free energy with overlap constraint. These will be of great use in the next section, where we develop their analogues at zero temperature. For any inverse temperature $\beta>0$, define the free energy by
	\begin{align*}
	F_{N,\beta} &=\frac{1}{N\beta}\e\log\int_{S_N}\exp \beta H_N(\sigma)\lambda_N(d\sigma),
	\end{align*}
	where $\lambda_N$ is the uniform probability measure on $S_N.$
	For any measurable subset $A$ of $[-1,1],$ we set the coupled free energy as 
	\begin{align*}
	\CF_{N,\beta} (A)&=\frac{1}{N\beta}\e\log\int_{R_{1,2}\in A}\exp \beta\bigl( H_N(\sigma^1)+H_N(\sigma^2)\bigr)\lambda_N(d\sigma^1)\times\lambda_N(d\sigma^2).
	\end{align*} 
	Let $\mathcal{M}$ be the space of all $(b,x)$ for $b\in\mathbb{R}$ and $x$ a p.d.f. on $[0,1]$ such that $$
	\max\Bigl(1,\int_0^1\beta^2\xi''(s)x(s)ds\Bigr)<b.
	$$
	Define the Parisi functional by
	\begin{align}\label{parisi}
	\mathcal{P}_\beta (b,x)=\frac{1}{2\beta}\Bigl(\frac{\beta^2h^2}{b-d_\beta^x(0)}+\int_0^1\frac{\beta^2\xi''(q)}{b-d_\beta^x(q)}dq+b-1-\log b-\int_0^1q\beta^2\xi''(q)x(q)dq\Bigr)
	\end{align}
	for any $(b,x)\in\mathcal{M},$ where $$
	d_\beta^x(q):=\int_q^1\beta^2\xi''(s)x(s)ds.$$ The Parisi formula for the free energy states that
	\begin{theorem}[Parisi's formula for the free energy]\label{PF}
		\begin{align}\label{PF:eq1}
		\lim_{N\rightarrow\infty}F_{N,\beta} &=\inf_{(b,x)\in \mathcal{M}}\mathcal{P}_\beta (b,x).
		\end{align}
	\end{theorem}
	The Parisi formula was rigorously established by Talagrand \cite{Tal06} and extended to general mixture of the model by Chen \cite{Chen13}. 
	In \cite{Tal06}, it was known that the optimization problem on the right-hand side has a unique minimizer, denoted by $(b_{\beta,P} ,x_{\beta,P}).$ The probability measure $\mu_{\beta,P}$ induced by $x_{\beta,P}$ is called the Parisi measure. 
	
	The coupled free energy can be controlled by a two-dimensional extension of the Parisi functional. For $a\in [0,1]$, let $\mathcal{M}_a$ be the collection of all $(b,\lambda,x)$ such that $b,\lambda\in\mathbb{R}$ and 
	$$
	\max\Bigl(1,|\lambda|+\int_0^1\beta^2\xi''(s)x(s)ds\Bigr)<b,
	$$
	where $x$ is a function of the form $$
	x(s)=1_{[0,a)}(s)x_1(s)+1_{[a,1]}(s)x_2(s),\,\,s\in[0,1]
	$$
	for $x_1$ and $x_2$ two nonnegative and nondecreasing functions with right continuity on $[0,a)$ and $[a,1]$ respectively and $x_1(a-)\leq 2x_2(a)$ and $x_2(1)=1.$ Let $u\in [-1,1]$ be fixed. Set $\iota=1$ if $u\geq 0$ and $\iota=-1$ if $u<0.$ 
	Define
	\begin{align*}
	\mathcal{P}_{\beta,u}(b,\lambda,x) &=\frac{T_{\beta,u}(b,\lambda,x)}{\beta}+\frac{1}{\beta}\left\{
	\begin{array}{ll}
	\frac{\beta^2h^2}{b-\lambda- d_\beta^x(0)},&\mbox{if $u\in[0,1]$},\\
	\frac{\beta^2h^2}{b-\lambda-d_{\beta}^x(|u|)},&\mbox{if $u\in[-1,0)$},
	\end{array}\right.
	\end{align*}
	where
	\begin{align*}
	T_{\beta,u}(b,\lambda,x)&:=\log \sqrt{\frac{b^2}{b^2-\lambda^2}}+ \int_0^{|u|} \frac{\beta^2\xi''(q)}{b-\iota \lambda-d_\beta^x(q)} dq \\
	&\quad +\frac{1}{2} \int_{|u|}^1 \frac{\beta^2\xi''(q)}{b-\lambda-d_\beta^x(q)} dq + \frac{1}{2} \int_{|u|}^1 \frac{\beta^2\xi''(q)}{b+\lambda-d_\beta^x(q)} dq \\
	&\quad-\lambda u +b-1- \log b - \beta^2\int_0^1 q\xi''(q) x(q)dq.
	\end{align*}
	The following theorem gives the Guerra-Talagrand RSB bound for the coupled free energy.  
	\begin{theorem}[RSB bound for the coupled free energy] Let $u\in [-1,1].$ For any $(b,\lambda, x)\in \mathcal{M}_{|u|}$, we have
		\label{RSBB}
		\begin{align}\label{RSBB:eq1}
		\lim_{\varepsilon\downarrow 0}\limsup_{N\rightarrow\infty}\CF_{N,\beta} ((u-\varepsilon,u+\varepsilon))&\leq \mathcal{P}_{\beta,u}(b,\lambda,x).
		\end{align}
	\end{theorem}
	
	This bound was previously introduced in \cite{Tal06} in order to establish the Parisi formula \eqref{parisi}. One may find its higher dimensional extension addressing temperature chaos and ultrametricity in \cite{PT}. In addition, a version of \eqref{RSBB:eq1} devoted to chaos in disorder was developed in \cite{C15}.

	\section{Bounds for the maximum energies}
	
	We present analogous results of Theorems \ref{PF} and \ref{RSBB} for the maximum energy $\ME_N $ as well as the maximum coupled energy $\MCE_{N} .$
	
	\subsection{Parisi's formula and RSB bound for the maximum energies}
	
	Recall $\mathcal{K}$ from the paragraph before \eqref{CS}. For $\nu\in \mathcal{K}$, define
	\begin{align}\label{more:eq2}
	\hat{\nu}(s)=\int_s^1\xi''(r)\nu(dr),\,\,s\in[0,1].
	\end{align}
	Let $\mathcal{U}$ be the collection of all $(B,\nu)\in \mathbb{R}\times\mathcal{K}$ satisfying   
	$$
	\hat{\nu}(0)<B.
	$$
	Define the Parisi functional on $\mathcal{U}$ by
	\begin{align}
	\label{PME:eq2}
	\mathcal{P} (B,\nu)&=\frac{1}{2}\Bigl(\frac{h^2}{B-\hat{\nu} (0)}+\int_0^{1} \frac{\xi''(s)}{B-\hat{\nu} (s)}ds+B-\int_0^1s\xi''(s)\nu(ds)\Bigr).
	\end{align}
	Our first main result in this subsection states another expression of the maximum energy via the Parisi formula at zero temperature.
	
	\begin{theorem}[Parisi's formula for the maximum energy]
		\label{PME}
		\begin{align}
		\label{PME:eq1}
		\ME &=\inf_{(B,\nu)\in \mathcal{K}}\mathcal{P} (B,\nu).
		\end{align}
		Here the minimum of the right-hand side is uniquely achieved by $(B_P ,\nu_P )\in \mathcal{K}$, where $\nu_P$ is the minimizer in the Crisanti-Sommers formula \eqref{CS} and $B_P$ satisfies
		\begin{align}\label{PME:eq3}
		B_P=\hat{\nu}_P(0)+\frac{1}{\nu_P([0,1])}.
		\end{align}  
	\end{theorem}
	
	The following proposition provides a characterization for the optimizer $(B_P,\nu_P).$
	
	\begin{proposition}
		\label{prop1}
		Let $(B,\nu)\in \mathcal{U}.$
		Define
		\begin{align*}
		\bar f(s)&=\int_s^1f(r)\xi''(r)dr,
		\end{align*}
		where
		$$
		f(r):=\frac{h^2}{(B-\hat{\nu}(0))^2}+\int_0^r \frac{\xi''(s)ds}{(B-\hat{\nu}(s))^2}-r.
		$$
		Then $(B,\nu)$ is the minimizer of $\mathcal{P}$ if and only if $f(1)=0$,  $\min_{r\in[0,1]}\bar{f}(r)\geq 0,$ and $\rho(S)=\rho([0,1))$, where $S:=\{r\in[0,1):\bar{f}(r)=0\}$, and $\rho$ is the measure on $[0,1)$ induced by $\gamma,$ i.e., $\rho([0,s])=\gamma(s)$ for $s\in[0,1).$ 
	\end{proposition}

	Next, we proceed to state the RSB bound for the maximum coupled energy. For $a\in[0,1],$ let $\mathcal{K}_a$ be the collection of all measures on $[0,1]$ of the form 
	$$
	\nu(ds)=1_{[0,a)}(s)\gamma_1(s)ds+1_{[a,1)}(s)\gamma_2(s)ds+\Delta\delta_{\{1\}}(ds)
	$$ 
	Here, $\gamma_1$ and $\gamma_2$ are  nonnegative and nondecreasing functions with right continuity on $[0,a)$ and $[a,1)$ respectively and they satisfy $\gamma_1(a-)\leq 2\gamma_2(a)$. Also, $0\leq \Delta<\infty$. For $\nu \in\mathcal{K}_a$, define $\hat{\nu}$ by \eqref{more:eq2}.
	Let $\mathcal{U}_a$ be the collection of all pairs $(B,\nu,\lambda)\in [0,\infty)\times \mathcal{K}_a\times\mathbb{R}$ with $$
	|\lambda|+\int_0^1\xi''(r)\nu(dr)<B.
	$$
	For any $u\in [-1,1]$, define the functional $\mathcal{P}_u $ on $\mathcal{U}_{|u|}$ by  
	\begin{align*}
	\mathcal{P}_u(B,\lambda,\nu) &=\int_0^{|u|} \frac{\xi''(q)}{B-\iota \lambda-\hat{\nu} (q)}dq+\frac{1}{2}\int_{|u|}^1\frac{\xi''(s)}{B-\lambda-\hat{\nu}(s)}ds+\frac{1}{2}\int_{|u|}^1\frac{\xi''(s)}{B+\lambda-\hat{\nu}(s)} ds\\
	&-\lambda u+B-\int_0^1s\xi''(s)\nu(ds)+\left\{
	\begin{array}{ll}
	\frac{h^2}{B-\lambda-\hat{\nu} (0)},&\mbox{if $u\in[0,1]$},\\
	\frac{h^2}{B-\lambda-\hat{\nu} (|u|)},&\mbox{if $u\in[-1,0)$}.
	\end{array}\right.
	\end{align*}
	Our RSB bound for the maximum coupled energy is stated as follows.

	\begin{theorem}[RSB bound for the maximum coupled energy]\label{lem5} Let $u\in[-1,1].$ For any $(B,\lambda,\nu)\in \mathcal{U}_{|u|}$, we have
		\begin{align}\label{lem5:eq1}
		\lim_{\varepsilon\downarrow 0}\limsup_{N\rightarrow\infty}\MCE_{N} \bigl((u-\varepsilon,u+\varepsilon)\bigr)&\leq \mathcal{P}_u(B,\lambda,\nu).
		\end{align}

	\end{theorem}

	One may find a similar inequality in Arnab-Chen \cite[Theorem 6]{ArnabChen15}, where \eqref{lem5:eq1} was shown to be valid along a special choice of the parameter $(B,\nu).$ In next sections, Theorem \ref{lem5} plays an essential role in controlling the maximum coupled energy by choosing proper parameter $(B,\lambda,\nu)$.

	\subsection{Proof of Theorems \ref{PME}, \ref{lem5} and Proposition \ref{prop1}} 
	
	\begin{proof}[\bf Proof of Theorem \ref{PME}] Let $(B,\nu)\in \mathcal{U}$. Let $\gamma$ be the density of $\nu$ on $[0,1)$ and $\Delta$ be the mass at $1.$ First, we assume that $\gamma(1-)<\infty.$ For $\beta>0$, let $b_\beta=B/\beta$ and define
		\begin{align*}
		x_{\beta}(s)&=\frac{\gamma(s)}{\beta}1_{[0,1-\Delta/\beta)}(s)+1_{[1-\Delta/\beta,1]}(s).
		\end{align*} 
		The assumption $\gamma(-1)<\infty$ guarantees that $(b_\beta,x_{\beta})\in \mathcal{U}$ for $\beta$ sufficiently large. Thus, a direct computation gives
		\begin{align*}
		\lim_{\beta\rightarrow\infty}\mathcal{P}_\beta (b_\beta,x_\beta)=\mathcal{P} (B,\nu).
		\end{align*}
		On the other hand, it is well-known (see e.g. \cite[Theorem 4,1]{AB} or \cite[Lemma 6]{ArnabChen15}) that
		\begin{align*}
		\ME=\lim_{N\rightarrow\infty}\ME_{N} &=\lim_{\beta\rightarrow\infty}\lim_{N\rightarrow\infty}F_N .
		\end{align*}
		Using \eqref{PF:eq1}, we obtain that 
		\begin{align*}
		\lim_{N\rightarrow\infty}\ME_N &\leq \mathcal{P} (B,\nu).
		\end{align*}
		One can easily release the assumption $\gamma(1-)<\infty$ by an approximation argument and consequently,
		\begin{align*}
		\ME &\leq \inf_{(B,\nu)\in\mathcal{U}}\mathcal{P} (B,\nu).
		\end{align*}
		To see that the equality holds, we recall the optimizer $(b_{\beta,P},x_{\beta,P})$ from \eqref{PF:eq1}. If we can show that $\mathcal{P}_\beta (b_{\beta,P},x_{\beta,P})$ converges to $\mathcal{P} (B,\nu)$ for certain $(B,\nu)\in \mathcal{U}$, then the Parisi formula \eqref{PF} together with the above inequality completes our proof. This part of the derivation has appeared in the work \cite{ArnabChen15}, where from Theorem 1, Lemma 7, and Equation (78) therein, it is known that there exists a sequence $(\beta_k)_{k\geq 1}$ with $\lim_{k\rightarrow\infty}\beta_k=\infty$ such that 
		\begin{align*}
		B_P &:=\lim_{k\rightarrow\infty}\beta_k^{-1}b_{\beta_k,P},\\
		\nu_P &=\lim_{k\rightarrow\infty}\beta x_{\beta_k,P}(s)ds\,\,\mbox{vaguely},\\
		B_P &>\int_0^1\xi''(s)\nu_P (ds),
		\end{align*} 
		and more importantly,
		\begin{align*}
		\ME=\lim_{N\rightarrow\infty}\ME_N =\mathcal{P} (B_P ,\nu_P ).
		\end{align*}
		This means that $(B_P ,\nu_P )\in \mathcal{U}$ and the announced formula holds. To see \eqref{PME:eq3}, we note that it was already established in the proof of \cite[Lemma 10]{ArnabChen15}.
	\end{proof}
	
	\begin{proof}[\bf Proof of Proposition \ref{prop1}]
		Assume that $(B,\nu)$ is the minimizer. Let $(B',\nu')$ be an arbitrary element in $\mathcal{U}.$ Write
		\begin{align*}
		\nu(ds)&=\gamma(s)1_{[0,1)}(s)ds+\Delta \delta_{\{1\}}(ds),\\
		\nu'(ds)&=\gamma'(s)1_{[0,1)}(s)ds+\Delta' \delta_{\{1\}}(ds).
		\end{align*}
		Let $\rho$ and $\rho'$ be the measures induced by $\gamma$ and $\gamma'.$ For $\theta\in[0,1]$, define $$
		(B_\theta,\nu_\theta)=(1-\theta)(B,\nu)+\theta(B',\nu').
		$$ 
		Then
		\begin{align}
		\begin{split}
		\label{prop1:proof:eq1}
		\mathcal{P}(B_\theta,\nu_\theta)\Big|_{\theta=0}&=\Bigl(-\frac{h^2}{(B-\hat{\nu}(0))^2}-\int_0^1\frac{\xi''(s)ds}{(B-\hat{\nu}(s))^2}+1\Bigr)(B'-B)\\
		&+\frac{h^2(\hat{\nu}'(0)-\hat{\nu}(0))}{(B-\hat{\nu}(0))^2}+\int_0^1\frac{\xi''(s)(\hat{\nu}'(s)-\hat{\nu}(s))}{(B-\hat{\nu}(s))^2}ds-\int_0^1s\xi''(s)(\nu'-\nu)(ds)\geq 0.
		\end{split}
		\end{align}
		From the first line of \eqref{prop1:proof:eq1}, $f(1)=0.$
		On the other hand, noting that
		\begin{align*}
		\frac{h^2(\hat{\nu}'(0)-\hat{\nu}(0))}{(B-\hat{\nu}(0))^2}=\int_0^1\frac{h^2\xi''(r)({\nu}'-{\nu})(dr)}{(B-\hat{\nu}(0))^2}
		\end{align*}
		and by Fubini's theorem,
		\begin{align*}
		\int_0^1\frac{\xi''(s)(\hat{\nu}'(s)-\hat{\nu}(s))}{(B-\hat{\nu}(s))^2}ds&=\int_0^1\int_0^r \frac{\xi''(s)ds}{(B-\hat{\nu}(s))^2}\xi''(r)({\nu}'-{\nu})(dr)\\
		\end{align*}
		the second line leads to			\begin{align*}
		\int_0^1f(r)\xi''(r)(\nu'-\nu)(dr)\geq 0.
		\end{align*}
		From this, Fubini's theorem yields
		\begin{align*}
		0&\leq\int_0^1 f(r)\xi''(r)(\nu'-\nu)(dr)\\
		&=\int_0^1f(r)\xi''(r)(\gamma'(r)-\gamma(r))dr+f(1)\xi''(1)(\Delta'-\Delta)\\
		&=\int_0^1\int_s^1f(r)\xi''(r)dr(\rho'-\rho)(ds)+f(1)\xi''(1)(\Delta'-\Delta).
		\end{align*}
		The validity of this inequality is equivalent to that $f(1)=0,$ $\min_{r\in[0,1]}\bar{f}(r)\geq 0,$ and $\rho(S)=\rho([0,1))$. 
	\end{proof}

	The proof of Theorem \ref{lem5} follows a similar argument as Theorem \ref{PME}.
	
	\begin{proof}[\bf Proof of Theorem \ref{lem5}]
		First we assume that $u\in (-1,1).$ Let $\varepsilon\in (0,1-|u|).$ An argument similar to \cite[Lemma 8]{ArnabChen15} leads to that for any $\beta>0,$
		\begin{align}\label{lem5:proof:eq1}
		\lim_{N\rightarrow\infty}\MCE_{N} ((u-\varepsilon/2,u+\varepsilon/2))\leq \limsup_{\beta\rightarrow\infty}\limsup_{N\rightarrow\infty}F_{N,\beta} ((u-\varepsilon,u+\varepsilon)).
		\end{align} 
		To bound the limit on the right-hand side, we use \eqref{RSBB} combined with a covering argument (see for instance \cite[Theorem 6]{ArnabChen15}) to obtain that for any $(b,\lambda,x)\in \mathcal{M}_{|u|},$
		\begin{align}\label{lem5:proof:eq2}
		\limsup_{N\rightarrow\infty}F_{N,\beta} ((u-\varepsilon,u+\varepsilon))&\leq \sup_{v\in (u-\varepsilon,u+\varepsilon)}\frac{\mathcal{P}_{\beta,v}(b,\lambda,x)}{\beta}.
		\end{align}
		Consider an arbitrary $(B,\lambda,\nu)\in \mathcal{U}_{|u|}$ for
		$$
		\nu(ds)=1_{[0,|u|)}(s)\gamma_1(s)ds+1_{[|u|,1)}(s)\gamma_2(s)ds+\Delta\delta_{\{1\}}(ds).
		$$ 
		For any $v\in [u-\varepsilon,u+\varepsilon]$, set
		$$
		\nu_v(ds)=1_{[0,|v|)}(s)\gamma_1(s)ds+1_{[|v|,1)}(s)\gamma_2(s)ds+\Delta\delta_{\{1\}}(ds).
		$$ 
		Assume that $\gamma_2(1-)<\infty.$ Set
		\begin{align*}
		b_\beta&=\beta B,\,\,\lambda_\beta=\beta\lambda
		\end{align*}
		and
		\begin{align*}
		x_{\beta,v}(s)&=1_{[0,|v|)}(s)\frac{\gamma_1(s)}{\beta}+1_{[|v|,1-\Delta/\beta)}(s)\frac{\gamma_2(s)}{\beta}+1_{[1-\Delta/\beta,1]}(s).
		\end{align*}
		Then $(b_\beta,\lambda_\beta,x_{\beta,v})\in \mathcal{M}_{|v|}$ for $\beta$ sufficiently large. As a result, a direct computation leads to
		\begin{align*}
		\lim_{\beta\rightarrow\infty}\frac{\mathcal{P}_{\beta,v}(b_\beta,\lambda_\beta,x_{\beta,v})}{\beta}=\mathcal{P}_v(B,\lambda,\nu_v).
		\end{align*}
		A key fact here is that this convergence is uniform over all $v\in [u-\varepsilon,u+\varepsilon].$
		This together with \eqref{lem5:proof:eq1} and \eqref{lem5:proof:eq2} implies
		\begin{align*}
		\lim_{N\rightarrow\infty}\MCE_{N} ((u-\varepsilon/2,u+\varepsilon/2))&\leq \sup_{v\in[u-\varepsilon,u+\varepsilon]}\mathcal{P}_v (B,\lambda,\nu_v).
		\end{align*} 
		Letting $\varepsilon\downarrow 0$ yields that
		\begin{align}\label{lem5:proof:eq3}
		\lim_{\varepsilon\downarrow 0}\lim_{N\rightarrow\infty}\MCE_{N} ((u-\varepsilon,u+\varepsilon))&\leq \mathcal{P}_u(B,\lambda,\nu).
		\end{align}
		By an approximation argument, we can release the assumption $\gamma_2(1-)<\infty$ and this inequality remains valid. To see how this inequality is also true for $u=\pm1$, we note that 
		\begin{align*}
		\lim_{u\rightarrow 1^-}\mathcal{P}_u (B,\lambda,\nu)&=\mathcal{P}_1 (B,\lambda,\nu),\\
		\lim_{u\rightarrow -1^+}\mathcal{P}_u(B,\lambda,\nu)&=\mathcal{P}_{-1} (B,\lambda,\nu).
		\end{align*}
		On the other hand, using Dudley's entropy integral, we can show that 
		\begin{align*}
		\lim_{u\rightarrow 1^-}\lim_{\varepsilon\downarrow 0}\limsup_{N\rightarrow\infty}\MCE_{N} (u-\varepsilon,u+\varepsilon)&=\lim_{\varepsilon\downarrow 0}\limsup_{N\rightarrow\infty}\MCE_{N} ((1-\varepsilon,1+\varepsilon)),\\
		\lim_{u\rightarrow -1^+}\lim_{\varepsilon\downarrow 0}\limsup_{N\rightarrow\infty}\MCE_{N} (u-\varepsilon,u+\varepsilon)&=\lim_{\varepsilon\downarrow 0}\limsup_{N\rightarrow\infty}\MCE_{N} ((-1-\varepsilon,-1+\varepsilon)).
		\end{align*}
		For detailed argument of this, we refer the readers to \cite[Lemma 13]{ArnabChen15}. Finally, our proof is completed by these inequalities and \eqref{lem5:proof:eq3}.
	\end{proof}

	\section{Control of maximum coupled energy} 
	
	In this section, we present the proof of Theorems \ref{thm1}, \ref{thm2} and \ref{thm1.1}, which is based on a subtle control of the RSB bound in the foregoing section.
	\subsection{Proof of Theorem \ref{thm1}}
	
	Recall the measure $\rho_P$ from \eqref{rho}. The proof of Theorem \ref{thm1} is a  consequence of the following theorem.
	
	\begin{theorem}
		\label{thm3} If $u\in \mbox{supp}\rho_P $, then for any $\varepsilon>0,$
		\begin{align*}
		\lim_{N\rightarrow\infty}\MCE_N \bigl((u-\varepsilon,u+\varepsilon)\bigr)&=2\ME .
		\end{align*}
	\end{theorem}
	
	\begin{proof}[\bf Proof of Theorem \ref{thm1}]
		The assumption $h=0$ implies that $H_{N}(\sigma)=H_N(-\sigma)$ for all $\sigma\in S_N,$ from which 
		\begin{align*}
		\MCE_N \bigl((u-\varepsilon,u+\varepsilon)\bigr)=\MCE_N \bigl((-u-\varepsilon,-u+\varepsilon)\bigr)
		\end{align*} 
		for any $|u|\in \mbox{supp}\rho_P.$
		Thus, it suffices to prove  \eqref{thm1:eq1} only for $u\in \mbox{supp}\rho_P.$ From Theorem \ref{thm3},
		\begin{align*}
		\lim_{N\rightarrow\infty}\MCE_N \bigl((u-\varepsilon,u+\varepsilon)\bigr)&=2\lim_{N\rightarrow\infty}\ME_N .
		\end{align*} 
		For any $\eta>0,$ there exists $N_0$ such that
		\begin{align*}
		\MCE_N \bigl((u-\varepsilon,u+\varepsilon)\bigr)\geq 2 \ME-\eta
		\end{align*}
		for all $N\geq N_0.$ Consequently, using concentration of measure for the Gaussian extrema processes, there exists a positive constant $K$ independent of $N$ such that with probability at most $1-Ke^{-N/K}$, 
		\begin{align}\label{thm1:proof:eq2}
		\frac{1}{N}\max_{R_{1,2}\in(u-\varepsilon,u+\varepsilon)}\bigl(H_N(\sigma^1)+H_N(\sigma^2)\bigr)\geq 2\ME-\frac{\eta}{2}
		\end{align}
		and
		\begin{align}\label{thm1:proof:eq3}
		\max_{\sigma\in S_N}\frac{H_N(\sigma)}{N}\leq \ME+\frac{\eta}{4}
		\end{align}
		for all $N\geq N_0.$ Therefore, from \eqref{thm1:proof:eq2}, there exist $\sigma^1,\sigma^2$ with $R_{1,2}\in(u-\varepsilon,u+\varepsilon)$ such that
		\begin{align*}
		\frac{H_N(\sigma^1)+H_N(\sigma^2)}{N}\geq 2\ME-\frac{\eta}{2}.
		\end{align*} 
		If either $H_N(\sigma^1)\leq N(\ME-\eta)$ or $H_N(\sigma^2)\leq N(\ME-\eta)$, then from this inequality and \eqref{thm1:proof:eq3},
		\begin{align*}
		2\ME-\frac{3\eta}{4}=2\ME+\frac{\eta}{4}-\eta>\frac{H_N(\sigma^1)+H_N(\sigma^2)}{N}\geq 2\ME-\frac{\eta}{2},
		\end{align*}
		which forms a contradiction. Therefore, $\p_N(\eta,(u-\varepsilon,u+\varepsilon))\geq 1-Ke^{-N/K}$ for all $N\geq N_0$ and this clearly implies Theorem \ref{thm1} with an adjusted constant $K.$ 
		
	\end{proof}
	
	For the remainder of this section, we prove Theorem \ref{thm3}. Recall the Parisi formula in Theorem~\ref{PF} and the optimizer $(b_{\beta,P},x_{\beta,P} ).$
	Recall that $\mu_{\beta,P} $ is the measure induced by $x_{\beta,P} .$ We say that the mixed even $p$-spin model is {\it generic} if the linear span of $\{s^p:c_p\neq 0\,\,\mbox{for some $p\in 2\mathbb{N}$}\}\cup \{1\}$ is dense in $C[0,1].$ We need two crucial lemmas. Lemma \ref{lem1} below shows that the coupled free energy is twice of the original free energy if the overlap constraint lies in the support of $\mu_{\beta,P} .$
	
	\begin{lemma}\label{lem1} Consider the generic mixed even $p$-spin model.
		Let $u$ be in the support of $\mu_{\beta,P} $. For any $\varepsilon>0,$ we have
		\begin{align}\label{lem1:eq1}
		\lim_{N\rightarrow\infty}\CF_{N,\beta} \bigl((u-\varepsilon,u+\varepsilon)\bigr)&=2\lim_{N\rightarrow\infty}F_{N,\beta} .
		\end{align}
	\end{lemma}
	
	\begin{proof}
		The assumption that the model is generic guarantees that the limiting law of the overlap $|R_{1,2}|$ is given by the Parisi measure $\mu_{\beta,P}$ under the measure $\e\la\cdot\ra_{\beta} ,$ where $\la\cdot\ra_{\beta} $ is the Gibbs average with respect to the exponential weight $\exp \beta H_N(\sigma)\lambda_N(d\sigma).$ Let $u\in \mbox{supp}\mu_{\beta,P}$ and $\varepsilon>0$ be fixed. Note that the trivial bound holds,
		\begin{align*}
		\CF_{N,\beta} \bigl((u-\varepsilon,u+\varepsilon)\bigr)&\leq 2F_{N,\beta} .
		\end{align*}
		If \eqref{lem1:eq1} is not valid, then there exists some $\eta_0>0$ such that
		\begin{align*}
		\CF_{N,\beta} \bigl((u-\varepsilon,u+\varepsilon)\bigr)<2F_{N,\beta} -\eta_0.
		\end{align*}
		for infinitely many $N$. Consequently, using the Gaussian concentration of measure for both $\CF_{N,\beta} \bigl((u-\varepsilon,u+\varepsilon)\bigr)$ and $F_{N,\beta} $, there exists some constant $K$ independent of $N$ such that with probability at least $1-Ke^{-N/K}$,
		\begin{align*}
		&\frac{1}{\beta N}\log \int_{|R_{1,2}-u|<\varepsilon}\exp \beta\bigl(H_N(\sigma^1)+H_N(\sigma^2)\bigr)\lambda_N(d\sigma^1)\times \lambda_N(d\sigma^2)\\
		&<\frac{2}{\beta N}\log \int \exp \beta H_N(\sigma)\lambda_N(d\sigma)-\frac{\eta_0}{2}
		\end{align*}
		for $N$ sufficiently large. This inequality yields
		\begin{align*}
		\liminf_{N\rightarrow\infty}\e \bigl\la I\bigl(|R_{1,2}-u|< \varepsilon\bigr)\bigr\ra_{\beta} &\leq \lim_{N\rightarrow\infty}\bigl(e^{-\frac{\beta\eta_0N}{2}}+Ke^{-\frac{N}{K}}\bigr)=0.
		\end{align*}
		In other words, $u$ is not in the support of $\mu_{\beta} $, a contradiction. Thus, \eqref{lem1:eq1} must hold.
	\end{proof}
	
	Next, we prove that the result of Lemma \ref{lem1} remains valid for the maximum coupled energy.
	
	\begin{lemma}\label{lem2}
		Assume that the model is generic. If $u\in \mbox{supp}\rho_P ,$ then for any $\varepsilon>0,$
		\begin{align}
		\label{lem2:eq1}
		\lim_{N\rightarrow\infty}\MCE_N \bigl((u-\varepsilon,u+\varepsilon)\bigr)=2\ME .
		\end{align}
	\end{lemma}
	
	\begin{proof}
		Let $u$ be in the support of $\rho_P$ and $\varepsilon>0$ be fixed.
		Recall that $(\beta x_{\beta,P}(s)ds)_{\beta>0}$ converges to $\nu_P $ vaguely from \cite[Theorem 1]{ArnabChen15}. There exists $u_\beta\in \mbox{supp}\mu_{\beta,P} $ such that $\lim_{\beta\rightarrow\infty}u_\beta=u$. Using Dudley's entropy integral, we can approximate the maximum coupled energy via the coupled free energy,
		\begin{align}
		\begin{split}\label{lem2:proof:eq1}
		\CF_{N,\beta} ((u_\beta-\varepsilon/2,u_\beta+\varepsilon/2)\bigr)+o_1(N,\beta)&\leq \MCE_N \bigl((u-\varepsilon,u+\varepsilon)\\
		&\leq \CF_{N,\beta} ((u_\beta-2\varepsilon,u_\beta+2\varepsilon)\bigr)+o_2(N,\beta),
		\end{split}
		\end{align}
		where $o_i(N,\beta)$ satisfies $\lim_{\beta\rightarrow\infty}\lim_{N\rightarrow\infty}o_i(N,\beta)=0$ for $i=1,2.$
		Since a similar argument for this type of the inequality has already appeared in the appendix of \cite{ArnabChen15} with great detail, we omit the proof here. 
		From \eqref{lem1:eq1}, 
		\begin{align*}
		\lim_{\beta\rightarrow \infty}\lim_{N\rightarrow \infty}\CF_{N,\beta} \bigl((u_\beta-2\varepsilon,u_\beta+2\varepsilon)\bigr)=2\limsup_{\beta\rightarrow\infty}\lim_{N\rightarrow\infty}F_{N,\beta} =2\ME ,\\
		\lim_{\beta\rightarrow \infty}\lim_{N\rightarrow \infty}\CF_{N,\beta} \bigl((u_\beta-\varepsilon/2,u_\beta+\varepsilon/2)\bigr)=2\limsup_{\beta\rightarrow\infty}\lim_{N\rightarrow\infty}F_{N,\beta} =2\ME .
		\end{align*}
		These equations combined with \eqref{lem2:proof:eq1} lead to \eqref{lem2:eq1}.	
	\end{proof}

	\begin{proof}[\bf Proof of Theorem \ref{thm3}] Let $\xi$ and $h$ be fixed. Recall the optimizer $(B_P ,\nu_P )$ associated to $\xi$ and $h$ in Theorem \ref{PME}. For each $n\geq 1$, let $(c_{n,p})_{p\in 2\mathbb{N}}$ be a sequence satisfying $0<c_{n,p}$ and $|c_p-c_{n,p}|<2^{-n-p}$ for all $p\in 2\mathbb{N}$. Define $\xi_n(s)=\sum_{p\in 2\mathbb{N}}c_{n,p}s^p$. Let $X_{N,n}$ be the mixed even $p$-spin Hamiltonian corresponding to $\xi_n$ and set $$
		H_{N,n}(\sigma):=X_{N,n}(\sigma)+h\sum_{i=1}^N\sigma_i.
		$$
		Note that the assumption $c_{n,p}>0$ for all $n\in\mathbb{N}$ and $p\in 2\mathbb{N}$ guarantees that $H_{N,n}$ is generic. Denote by $(B_n,\nu_n)$ the optimizer  associated to $\xi_n$ and $h$ in Theorem \ref{PME}.
		
		We claim that there exists a subsequence $(B_{n_k},\nu_{n_k})_{k\geq 1}$ such that 
		\begin{align*}
		\lim_{k\rightarrow\infty}B_{n_k}&=B_P ,\\
		\lim_{k\rightarrow\infty}\nu_{n_k}&=\nu_P \,\,\mbox{vaguely on $[0,1]$}.
		\end{align*}
		Recall Theorem \ref{PF}. Denote by $(b_{\beta,n},x_{\beta,n})$ the optimizer the Parisi formula for the free energy associated to $\xi_n$ and $h$. Recall two key inequalities from \cite[Lemma 2]{ArnabChen15}, 
		\begin{align*}
		\beta x_{\beta,n}(s)\leq \frac{2\sqrt{\xi_n'(1)}}{\xi_n(1)-\xi_n(s)},\,\,s\in[0,1)
		\end{align*}
		and
		\begin{align*}
		\int_0^1\beta x_{\beta,n}(s)ds\leq 2\sqrt{\xi_n'(1)}\Bigl(\frac{1}{\xi_n(1)-\xi_n(1/2)}+\frac{1}{\xi_n'(1/2)}\Bigr).
		\end{align*}
		From \cite[Theorem 1]{ArnabChen15}, sending $\beta$ in these two inequalities to infinity yields
		\begin{align*}
		\gamma_{n}(s)\leq \frac{2\sqrt{\xi_n'(1)}}{\xi_n(1)-\xi_n(s)},\,\,\forall s\in[0,1)
		\end{align*}
		and
		\begin{align*}
		\nu_n([0,1])&\leq 2\sqrt{\xi_n'(1)}\Bigl(\frac{1}{\xi_n(1)-\xi_n(1/2)}+\frac{1}{\xi_n'(1/2)}\Bigr),
		\end{align*}
		where $\gamma_n$ is the density of $\nu_n$ on $[0,1).$ Since $|c_{n,p}-c_p|<2^{-n-p}$, the first inequality implies that $\gamma_n$ is uniformly bounded on any interval $[0,s]$ for $s\in(0,1)$ and the second inequality means that $\nu_n$ is a sequence of bounded measures on $[0,1].$ From these, we can pass to subsequences such that $(\gamma_{n_k})_{k\geq 1}$ converges to some $\gamma_0$ vaguely on $[0,1)$ and $(\nu_{n_k})_{k\geq 1}$ converges to some $\nu_0$ vaguely on $[0,1]$, where
		$$
		\nu_0(ds)=\gamma_0(s)1_{[0,1)}(s)ds+1_{\{1\}}(s)\nu_0(\{1\})
		$$
		For $n\geq 0,$ define
		\begin{align*}
		\mathcal{P}_n(B,\nu)&=\frac{h^2}{B-\hat{\nu}_n (0)}+\int_0^{1} \frac{\xi_n''(s)}{B-\hat{\nu}_n (s)}ds+B-\int_0^1s\xi_n''(s)\nu_n(ds),
		\end{align*}
		where $\hat{\nu}_n(s):=\int_q^1\xi_n''(r)\nu(dr)$ for $q\in[0,1].$ Recall that $\ME$ is the limiting maximum energy of $H_N$ associated to $\xi$ and $h.$ Denote by $\ME_{n}$ the maximum energy of $H_{N,n}$ associated to $\xi_n$ and $h.$
		From the weak convergence of $(\nu_{n_k})_{k\geq 1}$ and Fatou's lemma,
		\begin{align}
		\begin{split}
		\label{thm1:proof:eq1}
		\ME  &=\lim_{k\rightarrow\infty}\ME_{n_k}\\
		&=\lim_{k\rightarrow\infty}\mathcal{P}_{n_k}(B_{n_k},\nu_{n_k})\\
		&\geq \frac{h^2}{B_0-\hat{\nu}_0 (0)}+\int_0^{1} \frac{\xi''(q)}{B_0-\hat{\nu}_0 (q)}dq+B_0-\int_0^1s\xi''(s)\nu_0(ds)\\
		&=\mathcal{P} (B_0,\nu_0),
		\end{split}
		\end{align}
		where $B_0:=\limsup_{k\rightarrow\infty}B_{n_k}.$ Note that $(B_n,\nu_n)\in \mathcal{U}$. This means $\hat{\nu}_n(0)<B_n$ for all $n\geq 1.$ It follows that $\hat{\nu}_0(0)\leq B_0$. Now from the first and third terms of the third line of \eqref{thm1:proof:eq1}, we can further conclude that $\hat{\nu}_0(0)<B_0<\infty$. In other words, $(B_0,\nu_0)\in \mathcal{U}.$
		Consequently, from the Parisi formula for $\ME $ in Theorem \ref{PME}, \eqref{thm1:proof:eq1} implies that $(B_0,\nu_0)$ is a minimizer and thus, $(B_0,\nu_0)=(B_P ,\nu_P ).$ This finishes the proof of our claim.
		
		Next, let $u\in\mbox{supp}\nu_P.$ Recall that $\MCE_N$ is the maximum coupled energy corresponding to $\xi$ and $h.$ Denote by $\MCE_{N,n}$ the maximum couped energy associated to $\xi_n$ and $h$. Using the subsequence $(\nu_{n_k})_{k\geq 1}$ obtained in the previous claim, we pick $u_{k}\in \mbox{supp}\nu_{n_k}$ such that $\lim_{k\rightarrow\infty}u_{k}=u.$ From this,
		\begin{align*}
		\lim_{N\rightarrow\infty}\MCE_N \bigl((u-\varepsilon,u+\varepsilon)\bigr)
		&=\lim_{k\rightarrow\infty}\lim_{N\rightarrow\infty}\MCE_{N,n_k}\bigl((u_k-\varepsilon,u_k+\varepsilon)\bigr)\\
		&=2\lim_{k\rightarrow\infty}\ME_{n_k}\\
		&=2\ME  ,
		\end{align*}
		where the first and third equalities hold since $(c_{n,p})_{n\geq 1}$ converges to $c_p$ uniformly over $p$ and the second equality used Lemma \ref{lem2}.
		This completes our proof.
	\end{proof}

	\subsection{Proof of Proposition \ref{prop0} and Theorem \ref{thm2}}
	
	Recall the constant $s_P$ from \eqref{GP:eq1}. Define
	\begin{align}\label{lem3:eq1}
	c(u)&=\nu_P\bigl([0,1]\bigr)^2\bigl(h^2+\xi'(u)\bigr)-u
	\end{align}
	for $u\in[-s_P,s_P].$ Recall the functions $f$ and $\bar f$ from Proposition \ref{prop1} associated to the minimizer $(B_P,\nu_P)$. We first establish a crucial lemma.
	\begin{lemma}
		\label{lem3}
		We have that
		\begin{align}
		\label{lem3:eq2}
		c(u)=f(u)
		\end{align}
		on $[0,1].$
		In addition,
		\begin{align}\label{thm2:proof:eq2}
		c(s_P)=0
		\end{align}
		and if $s_P\in(0,1]$, then
		\begin{align}\label{thm2:proof:eq3}
		c'(s_P)\leq 0.
		\end{align}
	\end{lemma}
	
	\begin{proof}
	 From $\nu_P([0,s_P))=0$ and \eqref{PME:eq3}, \eqref{lem3:eq2} holds. To see \eqref{thm2:proof:eq2} and \eqref{thm2:proof:eq3}, if $\mbox{supp}\rho_P=\emptyset,$ then $s_P=1.$ Since in this case the Parisi measure is replica symmetric,  we obtain \eqref{thm2:proof:eq2} from \eqref{RS:eq1}.
	    On the other hand, the discussion before Theorem \ref{thm:RS} implies \eqref{thm:RS1}. From this, \eqref{thm2:proof:eq3} follows since
		\begin{align*}
		c'(1)&=\frac{\xi''(1)}{\xi'(1)+h^2}-1\leq 0.
		\end{align*}
		
		 Next, if $\mbox{supp}\rho_P\neq \emptyset,$ then $s_P\in \mbox{supp}\rho_P$ and $\bar{f}(s_P)=0$ from Proposition~\ref{prop1}. In the case when $s_P\in (0,1)$, the optimality of $s_P$ implies $c(s_P)=f(s_P)=-\bar f'(s_P)=0$ and also $c'(s_P)=-\bar{f}''(s_P)\leq 0.$ These give \eqref{thm2:proof:eq2} and \eqref{thm2:proof:eq3}. If $s_P=0,$ then again by optimality of $s_P,$
		$$
		0\leq \bar{f}'(s_P)=-f(s_P)=-\frac{h^2}{(B_P-\hat{\nu}_P(0))^2}.
		$$
		This inequality holds only when $h=0$, from which $c(s_P)=f(s_P)=0.$ This completes our proof.
		
	\end{proof}
	
    \begin{proof}[\bf Proof of Proposition \ref{prop0}]
    Assume that $h=0.$ If $\xi(s)=s^2,$ then $s_P=1$ by \eqref{RS:eq1}. Suppose that $c_p>0$ for at least one even $p\geq 4.$ If $s_P>0,$ then from \eqref{thm2:proof:eq3},
    \begin{align*}
    c'(u)&=\nu_P\bigl([0,1]\bigr)^2\xi''(u)-1\leq \nu_P\bigl([0,1]\bigr)^2\xi''(s_P)-1=c'(s_P)\leq 0
    \end{align*}
    for $u\in[0,s_P].$ Since evidently $c(0)=0$ and $c(s_P)=0$ from \eqref{thm2:proof:eq2},  the above inequality implies that $c(u)=0$ on $[0,s_P].$ However, since $\xi$ is analytic on $(-1,1),$ this forces that 
    $$
    \xi'(u)=\frac{u}{\nu_P\bigl([0,1]\bigr)^2}
    $$
    on $[0,1]$, which contradicts the assumption. This completes the proof of Proposition \ref{prop0}$(i).$ As for Proposition \ref{prop0}$(ii)$, it can be easily obtained by noting that $s_P$ must satisfy $c(s_P)=0$ by \eqref{thm2:proof:eq2} and that $c(0)>0.$ 
    
    \end{proof}
	
	\begin{proof}[\bf Proof of Theorem \ref{thm2}]  From the assumption $h\neq 0,$ $s_P>0$ by the above remark. The statement of Theorem \ref{thm2}$(i)$ follows immediately via an identical reasoning as the proof of Theorem \ref{thm1} gives \eqref{thm2:eq1}. As for the proof of Theorem \ref{thm2}$(ii)$, it relies on the statement that for any $0<\varepsilon_0<s_P$, there exists some $\eta>0$ such that for every $u\in[-1,s_P-\varepsilon_0]$,
		\begin{align}\label{thm2:proof:eq0}
		\lim_{\varepsilon\downarrow 0}\limsup_{N\rightarrow\infty}\MCE_N \bigl((u-\varepsilon,u+\varepsilon)\bigr)\leq 2\ME -\eta.
		\end{align} 
		If this is valid, a standard covering argument (see, e.g., \cite{C15} or \cite{Tal06}) yields Theorem \ref{thm2}$(ii).$ Indeed, from \eqref{thm2:proof:eq0}, for any $u\in[-1,s_P-\varepsilon_0],$ there exist $\varepsilon_u>0$ and $N_u\geq 1$ such that 
		\begin{align}\label{thm2:proof:eq-1}
		\MCE_N \bigl((u-\varepsilon_u,u+\varepsilon_u)\bigr)\leq 2\ME -\frac{\eta}{2}
		\end{align} 
		for all $N\geq N_u.$ Since $[-1,s_P-\varepsilon]$ is a compact set, it can be covered by $(u_i-\varepsilon_{u_i},u_i+\varepsilon_{u_i})$ for $i=1,\ldots,n$ for some $u_1,\ldots,u_n\in [-1,s_P-\varepsilon_0].$ Therefore, from \eqref{thm2:proof:eq-1}, 
		\begin{align*}
		\MCE_N\bigl([-1,s_P-\varepsilon_0]\bigr)\leq 2\ME-\frac{\eta}{2}.
		\end{align*}
		for all $N\geq N_0:=\max_{1\leq i\leq n}N_{u_i}.$ Next from concentration of measure for Gaussian extrema processes, there exists $K>0$ such that with probability at least $1-Ke^{-N/K}$,
		\begin{align}\label{thm2:proof:eq-2}
		\frac{1}{N}\max_{R_{1,2}\in [-1,s_P-\varepsilon_0]}\bigl(H_N(\sigma^1)+H_N(\sigma^2)\bigr)\leq 2\ME-\frac{\eta}{4}.
		\end{align}
		If there exist $\sigma^1,\sigma^2$ such that $R_{1,2}\in [-1,s_P-\varepsilon_0]$, $H_N(\sigma^1)\geq N(\ME-\eta/16)$, and $H_N(\sigma^2)\geq N(\ME-\eta/16)$, then
		\begin{align*}
		\frac{H_N(\sigma^1)+H_N(\sigma^2)}{N}\geq 2\ME-\frac{\eta}{8}.
		\end{align*}
		From \eqref{thm2:proof:eq-2}, this means that $\p_N(\eta/16,[-1,s_P-\varepsilon_0])\leq Ke^{-N/K}$ for all $N\geq N_0$ and this clearly implies \eqref{thm2:eq2}.
		
		In what follows, we establish \eqref{thm2:proof:eq0} by four steps.

		{\bf Step 1:} We claim that there exists some $\eta_1>0$ such that for any $u\in [-s_P,s_P-\varepsilon_0],$
		\begin{align}\label{thm2:proof:eq7}
		\lim_{\varepsilon\downarrow 0}\limsup_{N\rightarrow\infty}\MCE_N \bigl((u-\varepsilon,u+\varepsilon)\bigr)&\leq 2\ME -\eta_1.
		\end{align}
		Note that for $u\in[-s_P,s_P]$, a direct differentiation yields that
		\begin{align*}
		\partial_\lambda\mathcal{P}_u (B_P ,\lambda,\nu_P )\Big|_{\lambda=0}&=\frac{h^2}{\bigl(B_P -\hat{\nu}_P  (0)\bigr)^2}+\int_0^{|u|} \frac{\iota\xi''(s)ds}{\bigl(B_P -\hat{\nu}_P  (s)\bigr)^2}-u=c(u),
		\end{align*}
		where $\iota$ is the sign of $u$ and $c(u)$ is defined in \eqref{lem3:eq1}. In addition, it can be easily derived that 
		\begin{align*}
		|\partial_{\lambda\lambda}\mathcal{P}_u (B_P ,\lambda,\nu_P )|&\leq M
		\end{align*}
		for all $|\lambda|\leq K:=(B_P -\hat{\nu}_P (0))/2.$
		Using Taylor's formula, for any $|\lambda|\leq K,$
		\begin{align*}
		\mathcal{P}_u (B_P ,\lambda,\nu_P )&\leq \mathcal{P}_u (B_P ,0,\nu_P )+c(u)\lambda+\frac{\eta \lambda^2}{2}.
		\end{align*}
		By varying $|\lambda|\leq K$ in this inequality, we can find $K'>0$ small enough such that for any $u\in [-s_P,s_P],$
		\begin{align}
		\begin{split}\label{thm2:proof:eq5}
		\lim_{\varepsilon\downarrow 0}\limsup_{N\rightarrow\infty}\MCE_N \bigl((u-\varepsilon,u+\varepsilon)\bigr)&\leq \min_{|\lambda|\leq K}\mathcal{P}_u (B_P ,\lambda,\nu_P )\\
		&\leq 2\ME -K'c(u)^2,
		\end{split}
		\end{align}
		where the first inequality used \eqref{lem5:eq1} and the second inequality relied on the fact that for $u\in[-s_P,s_P]$, 
		\begin{align*}
		\mathcal{P}_u (B_P ,0,\nu_P )=\mathcal{P} (B_P ,\nu_P )=2\ME .
		\end{align*}
		Note that $c'(u)\leq c'(s_P)\leq 0$ for $u\in[-s_P,s_P]$ by \eqref{thm2:proof:eq3}. This and \eqref{thm2:proof:eq2} together imply $$
		c(u)\leq c(s_P-\varepsilon)<c(s_P)=0$$ on $[-s_P,s_P-\varepsilon_0].$ From this and \eqref{thm2:proof:eq5}, \eqref{thm2:proof:eq7} follows with $\eta_1:=\min_{u\in[-s_P,s_P-\varepsilon_0]}K'c(u)^2>0$.
		
		{\bf Step 2:} We check that there exist some $\varepsilon_0'>0$ and $\eta_2>0$ such that for $u\in[-s_P-\varepsilon_0',-s_P]$, 	\begin{align}\label{thm2:proof:eq6}
		\lim_{\varepsilon\downarrow 0}\limsup_{N\rightarrow\infty}\MCE_N \bigl((u-\varepsilon,u+\varepsilon)\bigr)&\leq 2\ME -\eta_2.
		\end{align}
		Note that if $s_P=1$, then Step 1 completes our proof since the overlap satisfies $|R_{1,2}|\leq 1.$ In what follows, we assume that $s_P<1.$ Observe that $(u,\lambda)\mapsto\mathcal{P}_u (B_P ,\lambda,\nu_P )$ is continuous function on $[-1,0]\times [-K,K].$ We can choose $\varepsilon_0'>0$ small enough such that
		\begin{align*}
		\max_{u\in[-s_P-\varepsilon_0',-s_P]}\min_{|\lambda|\leq K}\mathcal{P}_u (B_P ,\lambda,\nu_P )\leq \min_{|\lambda|\leq K}\mathcal{P}_{-s_P} (B_P ,\lambda,\nu_P )+\frac{\eta_1}{2}.
		\end{align*}
		From \eqref{thm2:proof:eq5}, our claim \eqref{thm2:proof:eq6} is valid with $\eta_2=\eta_1/2.$
		
		{\bf Step 3:} Assume $u\in[-1,-s_P-\varepsilon_0']$. Letting $(B,\lambda,\nu)=(B_P ,0,\nu_P )$ in \eqref{lem5:eq1} yields
		\begin{align*}
		\mathcal{P}_u (B_P ,0,\nu_P )&=\int_0^{1} \frac{\xi''(q)}{B_P -\hat{\nu}_P  (q)}dq+B_P -\int_0^1s\xi''(s)\nu_P (ds)+
		\frac{h^2}{B-\hat{\nu}_P  (|u|)}.
		\end{align*}
		In view of Theorem \ref{PME}, 
		\begin{align}\label{thm2:proof:eq1}
		\lim_{\varepsilon\downarrow 0}\limsup_{N\rightarrow\infty}\MCE_N \bigl((u-\varepsilon,u+\varepsilon)\bigr)&\leq 2\ME -\eta_3,
		\end{align} 	
		where $\eta_3:=\min_{r\in[-1,-s_P-\varepsilon_0']}g(r)>0$ for
		$$
		g(r):=\frac{h^2\int_0^{|r|}\xi''(s)\nu_P (ds)}{\bigl(B_P -\hat{\nu}_P (0)\bigr)\bigl(B_P -\hat{\nu}_P (r)\bigr)}.
		$$
		
		{\bf Step 4:} Combining \eqref{thm2:proof:eq7}, \eqref{thm2:proof:eq6}, and \eqref{thm2:proof:eq1} and letting $\eta=\min(\eta_1,\eta_2,\eta_3)>0$ validate \eqref{thm2:proof:eq0}.	
			
	\end{proof}

	\subsection{Proof of Proposition \ref{thm1.1}}
	
	Our proof adapts an identical argument as \cite{CHL16}. The key ingredient is played by the so-called chaotic property in disorder for the maximum energy. For a fixed $k\in\mathbb{N},$ denote by $X_{N}^1,\ldots, X_{N}^k$ i.i.d. copies of $X_N.$ Let $t\in[0,1]$. Define the Hamiltonians
	\begin{align*}
	H_{N,t}^\ell(\sigma^\ell)&=\sqrt{t}X_N(\sigma^\ell)+\sqrt{1-t}X_{N}^\ell(\sigma^\ell)+h\sum_{i=1}^N\sigma_i^\ell
	\end{align*}
	for $\sigma^\ell\in S_N.$ For any measurable $A\subset [-1,1]$, we consider the maximum coupled energy,
	\begin{align*}
	\MCE_{N,t}(A)&:=\frac{1}{N}\e \max_{R(\sigma^\ell,\sigma^{\ell'})\in A}\bigl(H_{N,t}^\ell(\sigma^\ell)+H_{N,t}^{\ell'}(\sigma^{\ell'})\bigr).
	\end{align*}For $t\in[0,1],$ define
	\begin{align*}
	c_t(u)&=\nu_P\bigl([0,1]\bigr)^2\bigl(t\xi'(u)+h^2\bigr)-u
	\end{align*}
	on $[-s_P,s_P].$ Recall $c(u)$ from \eqref{lem3:eq1}. If $t=1,$ then $H_N=H_{N,t}^\ell$, $\MCE_N=\MCE_{N,t},$ and $c_1(u)=c(u).$ While Theorem \ref{thm3} says that the maximum coupled energy $\MCE_N$ converges to $2\ME$ if the overlap is restricted to any point in the support of the Parisi measure, chaos in disorder states that as long as $t\in(0,1)$,  $\MCE_{N,t}\bigl((u-\varepsilon,u+\varepsilon)\bigr)$ converges to $2\ME$ only if we take $u$ equal to a single point $u_{t}$. This result is established in Proposition 7 and Theorem 7 from \cite{ArnabChen15}, for which we recall as follows. 
	
	\begin{proposition}\label{lem4}
		Suppose $t\in(0,1)$ and $1\leq \ell<\ell'\leq k.$ For any $\varepsilon>0,$ there exists some $\eta>0$ such that
		\begin{align*}
		\lim_{N\rightarrow\infty}\MCE_{N,t}\bigl([-1,1]\setminus(u_t-\varepsilon,u_t+\varepsilon)\bigr)\leq 2\ME-\eta,
		\end{align*}
		where $u_t$ is the unique solution to $c_t(u)=0$ and it satisfies $u_t=0$ if $h=0$ and $u_t\in (0,s_P)$ if $h\neq 0.$	
	\end{proposition}
	
	\begin{lemma}\label{lem6}
	If $h\neq 0,$ then $t\mapsto u_t$ is continuous on $(0,1)$ with $\lim_{t\rightarrow 1-}u_t=s_P.$
	\end{lemma}
	
	\begin{proof}
	From Proposition \ref{lem4}, since $u_t>0$ and $\partial_t c_t(u)=\nu_P\bigl([0,1]\bigr)^2\xi'(u)>0$ for $u\in(0,1)$, the implicit function theorem implies that $u_t$ is continuous on $(0,1).$ If there exists $(t_n)\subset (0,1)$ such that $\lim_{n\rightarrow \infty}t_n=1$ and $v:=\lim_{n\rightarrow \infty}u_{t_n}<s_P$, then passing to the limit yields $$
	c(v)=\nu_P\bigl([0,1]\bigr)^2\bigl(\xi'(v)+h^2\bigr)-v=0.
	$$
	Note that $s_P>0$ since $h\neq 0.$ From \eqref{thm2:proof:eq3}, 
	$$
	c'(u)=\nu_P\bigl([0,1]\bigr)^2\xi''(u)-1\leq \nu_P\bigl([0,1]\bigr)^2\xi''(s_P)-1=c'(s_P)\leq 0.
	$$
	Consequently, the last two displays together with $c(s_P)=0$ deduce that $c(u)=0$ for all $u\in[v,s_P]$ or equivalently 
	\begin{align}
	\label{lem6:proof:eq1}
		\xi'(u)=\frac{u}{\nu_P\bigl([0,1]\bigr)^2}-h^2
	\end{align}
	on $[v,s_P].$ However, since $\xi$ is analytic in $(-1,1)$, this forces that \eqref{lem6:proof:eq1} holds for all $u\in[-1,1]$, a contradiction as $-h^2=\xi'(0)=0.$ 
		
	\end{proof}

	\begin{proof}[\bf Proof of Proposition \ref{thm1.1}] We first prove the case $h\neq 0.$  Let $\varepsilon,\eta>0.$ From Lemma \ref{lem6}, there exists $t_0$ such that 
		\begin{align}
		\label{thm1.1:proof:eq2}
		|u_t-s_P| <\frac{\varepsilon}{2}
		\end{align}
		whenever $t_0<t<1.$ Let $1\leq \ell<\ell'\leq k.$ Denote by $\sigma_t^\ell$ an maximizer of $H_{N,t}^\ell$ over $S_N$. Also denote by $L_N$ the maximum of $|X_N|$ over $S_N$ and by $L_N^\ell$ the maximum of $|X_N^\ell|$ over $S_N.$ Note that
		\begin{align*}
		\Bigl|\frac{H_{N,t}^\ell(\sigma_{t}^\ell)}{N}-\frac{H_{N}(\sigma_{t}^\ell)}{N}\Bigr|\leq \bigl(1-\sqrt{t}\bigr)\frac{L_N}{N}+\sqrt{1-t}\frac{L_N^\ell}{N}.
		\end{align*}
		Using the concentration of measure for Gaussian extrema processes $L_N$ and $L_N^\ell$, it can be show that there exists a constant $C>0$ independent of $t$ such that with probability at least $1-Ce^{-N/C}$,  
			\begin{align*}
			\Bigl|\frac{H_{N,t}^\ell(\sigma_{t}^\ell)}{N}-\frac{H_{N}(\sigma_{t}^\ell)}{N}\Bigr|\leq \bigl(1-\sqrt{t}+\sqrt{1-t}\bigr)C
			\end{align*}
			for all $t\in[0,1]$. Consequently, from \eqref{CS}, 
		\begin{align}
		\label{thm1.1:proof:eq1}\p\Bigl(\frac{H_N(\sigma_t^\ell)}{N}\geq \ME-\eta\Bigr)\leq Ce^{-\frac{N}{C}}
		\end{align}
		provided that $t$ is sufficiently close to $1$ such that 
		$$
		\bigl(1-\sqrt{t}+\sqrt{1-t}\bigr)C<\eta.
		$$
		From now on, we fix a $t>t_0$ for the rest of the proof.
	
	Next, from Proposition \ref{lem4}, one can argue in the same way as the proof of Theorems \ref{thm1}$(i)$ and \ref{thm2}$(ii)$ to show that there exists some $C'>0$ such that the probability 
	\begin{align}\label{thm1.1:proof:eq4}
	\p\Bigl(\bigl|R(\sigma_t^\ell,\sigma_t^{\ell'})-u_t\bigr|>\frac{\varepsilon}{2}\Bigr)\leq C'e^{-\frac{N}{C'}}
	\end{align} 
	for all $N\geq 1.$ From \eqref{thm1.1:proof:eq2},
	\begin{align}\label{thm1.1:proof:eq3}
	\p\Bigl(\bigl|R(\sigma_t^\ell,\sigma_t^{\ell'})-s_P\bigr|>\varepsilon\Bigr)\leq C'e^{-\frac{N}{C'}}.
	\end{align} 
	Note that all the estimates above are independent of $1\leq \ell<\ell'\leq k.$ Combining \eqref{thm1.1:proof:eq1} and \eqref{thm1.1:proof:eq3}, if we take $k$ to be the largest integer such that
	$$
	k\leq \exp\Bigl(\frac{N}{2\max(2C',C)}\Bigr)
	$$
	and $O_N=\{\sigma_t^1,\ldots,\sigma_t^k\}$, then Proposition \ref{thm1.1} follows for the case $h\neq 0.$	The case $h=0$ is easier since now $u_t=0$ for all $t\in (0,1).$ With this we can combine \eqref{thm1.1:proof:eq1} and \eqref{thm1.1:proof:eq4} directly to obtain the announced result.
	
	\end{proof}

	\section{Establishing energy landscapes} 
	
	We provide the proofs for Theorems \ref{thm:RS},  \ref{FRSB} and \ref{thm:1RSB}. The proof of Theorems \ref{thm:RS} and \ref{FRSB} are immediate consequence of Theorems \ref{thm1} and \ref{thm2}, while the verification of Theorem \ref{thm:1RSB} is based on the RSB bound in \eqref{lem5:eq1} with a careful choice of the parameters.
	
	\subsection{RS and FRSB solutions}
	
	The proofs of Theorems \ref{thm:RS} and \ref{FRSB} are immediate consequences of Theorems \ref{thm1} and \ref{thm2}.
	\begin{proof}[\bf Proof of Theorem \ref{thm:RS}]
		From \eqref{RS:eq1}, we see that $\rho_P=\emptyset$ and $\Gamma=\{1\}.$ Since $h\neq 0,$ Theorem \ref{thm:RS} follows from Theorem \ref{thm2}.
	\end{proof}

	\begin{proof}[\bf Proof of Theorem \ref{FRSB}]
		From \ref{eq2}, $\Gamma=[-1,1]$ if $h=0$ and $\Gamma=[s_P,1]$ if $h\neq 0.$ Theorem \ref{FRSB}$(i)$ follows from Theorem \ref{thm1}, while Theorem~\ref{FRSB}$(ii)$ is valid by Theorem \ref{thm2}. 
		
	\end{proof}

	\subsection{1RSB solution}
	
	First we develop an auxiliary lemma.
	Recall the functional $\mathcal{P}$ form \eqref{PME:eq2}, the constant $z$ from \eqref{1RSB:eq0}, and the function $\zeta$ from \eqref{1RSB:lem-1:eq1}.
	
	\begin{lemma}
		\label{lem0}
		Consider $(B,\nu)\in \mathcal{U}$ defined by
		\begin{align*}
		\nu(ds)&=A1_{[0,1)}(s)ds+\Delta \delta_{\{1\}}(ds),\\
		B&=\xi''(1)\Delta+\Delta^{-1},
		\end{align*}
		where
		\begin{align}
		\begin{split}\label{lem0:eq1}
		\delta&:=z(1+z)^{-1},\\
		A&:=\delta^{1/2}z^{1/2}\xi'(1)^{-1/2}=z(1+z)^{-1/2}\xi'(1)^{-1/2},\\
		\Delta&:=\delta^{1/2} z^{-1/2}\xi'(1)^{-1/2}=(1+z)^{-1/2}\xi'(1)^{-1/2}.
		\end{split}
		\end{align}
		Recall the two functions $f,\bar f$ in Proposition \ref{prop1} associated to $(B,\nu)$ and $h=0.$ Then
		\begin{align}
		\begin{split}\label{lem0:eq2}
		\bar{f}(s)&=-\zeta(s),	\,\,\forall s\in[0,1],
		\end{split}\\
		\begin{split}
		\label{lem0:eq3}
		f(1)&=0
		\end{split}
		\end{align}
		and
		\begin{align}\label{lem0:eq4}
		\mathcal{P}(B,\nu)&=\frac{\xi'(1)+z\xi(1)}{\sqrt{(1+z)\xi'(1)}}.
		\end{align}
	\end{lemma}
	
	\begin{proof}
		Since
		\begin{align*}
		B-\hat{\nu}(s)&=\Delta^{-1}-A(\xi'(1)-\xi'(s)),
		\end{align*}
		a direct computation gives
		\begin{align*}
		\int_0^s\frac{\xi''(r)dr}{\bigl(B-\hat{\nu}(r)\bigr)^2}&=-\frac{1}{A\bigl(\Delta^{-1}-A(\xi'(1)-\xi'(r))\bigr)}\Big|_0^s\\
		&=\frac{\Delta^2\xi'(s)}{\bigl(1-\Delta A(\xi'(1)-\xi'(s))\bigr)\bigl(1-\Delta A\xi'(1)\bigr)}\\
		&=\frac{\delta w(s)}{z(1-\delta)(1-\delta+\delta w(s))}
		\end{align*}
		for $w(s):=\xi'(s)/\xi'(1).$ Thus, \eqref{lem0:eq3} follows from
		\begin{align*}
		{f}(1)&=\int_0^1 \frac{\xi''(r)dr}{\bigl(B-\hat{\nu}(r)\bigr)^2}-1=\frac{\delta }{z(1-\delta)(1-\delta+\delta)}-1=0.
		\end{align*}
		In addition, we compute
		\begin{align*}
		&\int_u^1\int_0^s\frac{\xi''(r)dr}{\bigl(B-\hat{\nu}(r)\bigr)^2}\xi''(s)ds\\
		&=\frac{\delta}{(1-\delta)z}\int_u^1\frac{\delta w(s)\xi''(s)ds}{1-\delta+\delta w(s)}\\
		&=\frac{1}{(1-\delta)z}\int_u^1 \Bigl(1-\frac{1-\delta}{1-\delta+\delta w(s)}\Bigr)\xi''(s)ds\\
		&=\frac{1+z}{z}\int_u^1\Bigl(\xi''(s)-\frac{\xi''(s)}{1+\frac{z\xi'(s)}{\xi'(1)}}\Bigr)ds\\
		&=\frac{1+z}{z}\bigl(\xi'(1)-\xi'(u)\bigr)-\frac{(1+z)\xi'(1)}{z^2}\Bigl(\log \Bigl(1+\frac{z\xi'(s)}{\xi'(1)}\Bigr)\Big|_{u}^1\Bigr)\\
		&=\frac{1+z}{z}\bigl(\xi'(1)-\xi'(u)\bigr)-\frac{(1+z)\xi'(1)}{z^2}\Bigl(\log(1+z)-\log\Bigl(1+\frac{z\xi'(u)}{\xi'(1)}\Bigr)\Bigr),
		\end{align*}
		and
		\begin{align*}
		\int_u^1s\xi''(s)ds&=\xi'(1)-\xi'(u)u-\bigl(\xi(1)-\xi(u)\bigr).
		\end{align*}
		Combining these two equations together and applying \eqref{1RSB:eq0} yield
		\begin{align*}
		\bar{f}(u)&=\int_u^1 \Bigl(\int_0^s \frac{\xi''(r)dr}{\bigl(B-\hat{\nu}(r)\bigr)^2}-s\Bigr)\xi''(s)ds\\
		&=\xi(1)-\xi(u)-\xi'(1)+\xi'(u)u+\frac{1+z}{z}\bigl(\xi'(1)-\xi'(u)\bigr)\\
		&\quad-\frac{(1+z)\xi'(1)}{z^2}\Bigl(\log(1+z)-\log\Bigl(1+\frac{z\xi'(u)}{\xi'(1)}\Bigr)\Bigr)\\
		&=-\zeta(u).
		\end{align*}
		This gives \eqref{lem0:eq2}. As for \eqref{lem0:eq4}, it can be justified by
		\begin{align*}
		\mathcal{P} (B,\nu)
		&=\frac{1}{2}\Bigl(\Delta^{-1}-(\xi'(1)-\xi(1))A+\frac{1}{A}\log \frac{1}{1-\delta}\Bigr)\\
		&=\frac{1}{2}\Bigl(\sqrt{(1+z)\xi'(1)}-\frac{z(\xi'(1)-\xi(1))}{\sqrt{(1+z)\xi'(1)}}+\frac{\sqrt{(1+z)\xi'(1)}}{z}\log (1+z)\Bigr)\\
		&=\frac{1}{2\sqrt{(1+z)\xi'(1)}}\Bigl((1+z)\xi'(1)-z(\xi'(1)-\xi(1))+\frac{(1+z)\xi'(1)}{z}\log(1+z)\Bigr)\\
		&=\frac{1}{2\sqrt{(1+z)\xi'(1)}}\Bigl(\xi'(1)+z\xi(1)+z\xi'(1)\Bigl(\frac{\xi(1)}{\xi'(1)}+\frac{1}{z}\Bigr)\Bigr)\\
		&=\frac{\xi'(1)+z\xi(1)}{\sqrt{(1+z)\xi'(1)}}.
		\end{align*}
	\end{proof}
	
	\begin{proof}[\bf Proof of Theorem \ref{thm:1RSB}]
		Assume that the Parisi measure $\rho_P$ is 1RSB with $\mbox{supp}\rho_P=\{0\}.$ If $c_p=0$ for all even $p\geq 4,$ then $\xi(s)=s^2.$ In this case, we learn from \eqref{RS:eq1} that $\rho_P$ must be RS, a contradiction. Thus, $c_p>0$ for at least one even $p\geq 4.$  We prove that $\zeta\leq 0$ on $[0,1].$ Write 
		\begin{align}
		\label{thm:1RSB:proof:eq0}
		\nu_P(ds)=A_P1_{[0,1)}(s)ds+\Delta_P\delta_{\{1\}}(ds)
		\end{align}
		for some $A_P,\Delta_P>0.$ Recall the variational representation \eqref{CS}. It is known from \cite[Theorem 2]{ArnabChen15} that the optimality of $\nu_P$ in the Crisanti-Sommers formula yields the following two equations
		\begin{align*}
		\int_0^1\Bigl(\xi'(s)-\int_0^s\frac{dr}{\nu_P([r,1])^2}\Bigr)ds&=0,\\
		\int_0^1\frac{ds}{\nu_P([s,1])^2}&=\xi'(1).
		\end{align*}
		Plugging \eqref{thm:1RSB:proof:eq0} into these equations gives
		\begin{align}
		\begin{split}\label{thm:1RSB:proof:eq11}
		\frac{1}{A_P^2}\log\Bigl(1+\frac{A_P}{\Delta_P}\Bigr)-\frac{1}{A_P(A_P+\Delta_P)}&=\xi(1),
		\end{split}\\
		\begin{split}
		\label{thm:1RSB:proof:eq12}
		\frac{1}{\Delta_P(A_P+\Delta_P)}&=\xi'(1).
		\end{split}
		\end{align}
		A substitution of \eqref{thm:1RSB:proof:eq11}   by \eqref{thm:1RSB:proof:eq12} yields
		\begin{align*}
		\frac{\Delta_P (A_P+\Delta_P)}{A_P^2}\xi'(1)\log\Bigl(1+\frac{A_P}{\Delta_P}\Bigr)-\frac{\Delta_P}{A_P}\xi'(1)&=\xi(1).
		\end{align*}
		If we let $z=A_P/\Delta_P,$ then this equation coincides with \eqref{1RSB:eq0}. Furthermore, from \eqref{thm:1RSB:proof:eq12}, we obtain
		\begin{align*}
		A_P&=z(1+z)^{-1/2}\xi'(1)^{-1/2},\\
		\Delta_P&=(1+z)^{-1/2}\xi'(1)^{-1/2}.
		\end{align*} to get \eqref{thm1:1RSB:eq1} and \eqref{thm1:1RSB:eq2}. Now by comparing the two formulas \eqref{CS} and \eqref{PME:eq1} and letting $B=\xi''(1)\Delta_P+\Delta_P^{-1}$, since a direct verification gives
		\begin{align*}
		\mathcal{Q}(\nu_P)=\frac{\xi'(1)+z\xi(1)}{\sqrt{(1+z)\xi'(1)}}=\mathcal{P}(B,\nu_P),
		\end{align*}
		we see that $B_P=B$ by Theorem \ref{PME}. Next, recall the functions $f,\bar{f}$ associated to $(B_P,\nu_P)$ and $h=0$ from Proposition \ref{prop1}. Then Proposition \ref{prop1} and Lemma \ref{lem0} together imply that $-\zeta(s)=\bar{f}(s)\geq 0$ for all $s\in[0,1].$ This validates \eqref{thm:1RSB:eq1}.
		
		Conversely, assume that $c_p>0$ for at least one even $p\geq 4$ and \eqref{thm:1RSB:eq1} is valid. From Lemma~\ref{lem0}, recall the pair $(B,\nu)$ and note $\bar{f}(s)=-\zeta(s)$. From \eqref{thm:1RSB:eq1} and \eqref{lem0:eq2}, it follows that $\bar{f}(s)\geq 0$ on $[0,1]$ and $\bar{f}(0)=0$. In addition, $f(1)=0$ by \eqref{lem0:eq3} and $\rho(S)=A=\rho([0,1))$ since $0\in S$, where $S=\{s\in[0,1):f(s)=0\}$ and $\rho$ the measure induced by $\nu.$ These together imply that $(B,\nu)$ must be the minimizer of $\mathcal{P}$ by Proposition \ref{prop1}. This means that the Parisi measure is 1RSB with $\mbox{supp}\rho_P=\{0\}.$ Finally the validity of \eqref{thm1:1RSB:eq1} and \eqref{thm1:1RSB:eq2} follows by Lemma \ref{lem0}.
		
	\end{proof}
	
	\begin{proof}[\bf Proof of Theorem \ref{thm:1RSB:os}] Assume that $h=0$ and \eqref{thm:1RSB:os:eq1} holds. We verify the following inequality: For any $0<\varepsilon<1/2,$ there exists some $\eta>0$ such that
		\begin{align*}
		\limsup_{N\rightarrow\infty}\MCE_{N}\bigl([-1+\varepsilon,-\varepsilon]\cup[\varepsilon,1-\varepsilon]\bigr)<2\ME-\eta.
		\end{align*} 
		The validity of this inequality is equivalent to 
		Recall $(B,\nu)$ from \eqref{lem0:eq1} and $\delta,A,\Delta$ from \eqref{lem0:eq1}.
		Let $u\in[-1,1]$ with $\varepsilon\leq |u|\leq 1-\varepsilon.$ Denote $a=|u|$. For $0<m<2$, set 
		\begin{align*}
		\nu_m(ds)&=A\bigl(m 1_{[0,a)}(s)+1_{[a,1)}(s)\bigr)ds+\Delta \delta_{\{1\}}(ds).
		\end{align*}
		Note that for $s\in[0,1),$
		\begin{align*}
		B-	\hat{\nu}_m(s)&=\int_s^1\xi''(r)\nu_m(dr)\\
		&=\Delta^{-1}-A\bigl(m(\xi'(a)-\xi'(s))+\xi'(1)-\xi'(a)\bigr)1_{[0,a)}(s)-A\bigl(\xi'(1)-\xi'(s)\bigr)1_{[a,1)}(s).
		\end{align*}
		Then
		\begin{align*}
		&\int_0^{1} \frac{\xi''(s)}{B-\hat{\nu}_m (s)}ds\\
		&=\frac{1}{Am}\log\bigl(\Delta^{-1}-Am(\xi'(a)-\xi'(s))-A(\xi'(1)-\xi'(a))\bigr)\Big|_{0}^a+\log \bigl(\Delta^{-1}-A(\xi'(1)-\xi'(s)\bigr)\Bigl|_a^1\\
		&=\frac{1}{Am}\log \frac{\Delta^{-1}-A(\xi'(1)-\xi'(a))}{\Delta^{-1}-Am\xi'(a)-A(\xi'(1)-\xi'(a))}-\frac{1}{A}\log\frac{\Delta^{-1}-A(\xi'(1)-\xi'(a))}{\Delta^{-1}}\\
		&=\frac{1}{Am}\log \frac{1-A\Delta(\xi'(1)-\xi'(a))}{1-A\Delta m\xi'(a)-A\Delta(\xi'(1)-\xi'(a))}-\frac{1}{A}\log\bigl(1-A\Delta(\xi'(1)-\xi'(a)\bigr)\\
		&=\frac{1}{Am}\log \frac{1-\delta(1-w(a))}{1-\delta mw-\delta(1-w(a))}-\frac{1}{A}\log\bigl(1-\delta(1-w(a))\bigr)
		\end{align*}
		for $w(a):=\xi'(a)/\xi'(1)$. In addition,
		\begin{align*}
		\int_0^1s\xi''(s)\nu_m(ds)&=Am\int_0^as\xi''(s)ds+A\int_a^1s\xi''(s)ds+\xi''(1)\Delta\\
		&=Am(a\xi'(a)-\xi(a))+A\bigl(\xi'(1)-\xi(1)-(a\xi'(a)-\xi(a))\bigr)+\xi''(1)\Delta\\
		&=-A(a\xi'(a)-\xi(a))(1-m)+A(\xi'(1)-\xi(1))+\xi''(1)\Delta.
		\end{align*}
		From these two equations,
		\begin{align*}
		\mathcal{P}_u (B,0,\nu_m) &=\int_0^{1} \frac{\xi''(s)}{B-\hat{\nu}_m (s)}ds+B-\int_0^1s\xi''(s)\nu_m(ds)\\
		&=\Delta^{-1}+A(a\xi'(a)-\xi(a))(1-m)-A(\xi'(1)-\xi(1))\\
		&\quad    +\frac{1}{Am}\log \frac{1-\delta(1-w(a))}{1-\delta mw(a)-\delta(1-w(a))}-\frac{1}{A}\log\bigl(1-\delta(1-w(a))\bigr).
		\end{align*}
		In particular, if $m=1$, from the first equality and \eqref{thm1:1RSB:eq2}, 
		\begin{align*}
		\mathcal{P}_u(B,0,\nu_1)&=2\mathcal{P}(B,\nu)=2\ME.
		\end{align*}
		Next, 
		\begin{align*}
		\partial_m	\mathcal{P}_u (B,0,\nu_m)\Big|_{m=1}&=-A(a\xi'(a)-\xi(a))+\frac{1}{A}\log\frac{1-\delta}{1-\delta(1-w)}+\frac{\delta w(a)}{A(1-\delta)}\\
		&=A\Bigl(-(a\xi'(a)-\xi(a))+\frac{1}{A^2}\log\frac{1-\delta}{1-\delta(1-w(a))}+\frac{\delta w(a)}{A^2(1-\delta)}\Bigr)\\
		&=A\Bigl(-(a\xi'(a)-\xi(a))-\frac{\xi'(1)(1+z)}{z^2}\log(1+w(a)z)+\frac{\xi'(1)w(a)(1+z)}{z}\Bigr)\\
		&=A\zeta(a),
		\end{align*}
		where $\zeta$ is defined in \eqref{1RSB:lem-1:eq1}.
		Since $\zeta<0$ on $(0,1)$, 
		$$
		\partial_m	\mathcal{P}_u (B,0,\nu_m)\Big|_{m=1}<0.
		$$
		Since $(u,m)\mapsto\partial_m\mathcal{P}_u(B,0,\nu_m)$ is continuous on $\{u:\varepsilon\leq |u|\leq 1-\varepsilon\}\times [0,2]$, from the mean value theorem, there exist $m$ around $1$ and $\eta>0$ such that for any $u$ with$|u|\in[\varepsilon,1-\varepsilon]$,
		\begin{align*}
		\mathcal{P}_u (B,0,\nu_m)&\leq \mathcal{P}_u (B,0,\nu_1)-4\eta\\
		&=2\ME -4\eta.
		\end{align*}
		Therefore, from Theorem \ref{lem5}, for any $u$ satisfying $|u|\in[\varepsilon,1-\varepsilon],$
		\begin{align*}
		\lim_{\varepsilon'\downarrow 0}\limsup_{N\rightarrow\infty}\MCE_{N} \bigl((u-\varepsilon',u+\varepsilon')\bigr)<2\ME -4\eta.
		\end{align*}
		The assertion \eqref{1RSB:eq1} then follows by an identical argument as the proof of Theorem \ref{thm2}$(ii).$
		
	\end{proof}

	\begin{proof}[\bf Proof of Theorem \ref{cor1}] Recall $\zeta$ from \eqref{1RSB:lem-1:eq1}. Our goal is to show that $\zeta<0$ on $(0,1).$ Observe that $\zeta(0)=\zeta(1)=0$. Computing directly gives
		\begin{align*}
		\zeta'(s)&=\xi''(s)(1-s)-\frac{\xi'(1)(1+z)}{z}\frac{\xi''(s)}{\xi'(1)+z\xi'(s)}+\frac{\xi''(s)}{z}\\
		&=\frac{\xi''(s)}{z(\xi'(1)+z\xi'(s)}\bigl(z(1-s)(\xi'(1)+z\xi'(s))-\xi'(1)(1+z)+\xi'(1)+z\xi'(s)\bigr)\\
		&=-\frac{s\xi''(s)}{\xi'(1)+z\xi'(s)}\Bigl(\xi'(1)+z\xi'(s)-\frac{\xi'(s)}{s}(1+z)\Bigr)\\
		&=-\frac{\xi'(1)\xi'(s)\xi''(s)\phi(s)}{\xi'(1)+z\xi'(s)},
		\end{align*}
		where
		\begin{align*}
		\phi(s)&:=\frac{s}{\xi'(s)}+\frac{zs}{\xi'(1)}-\frac{1+z}{\xi'(1)},\,\,s\in[0,1].
		\end{align*}
		Observe that from \eqref{cor:eq1}, $$
		\phi(0)=\frac{1}{\xi''(0)}-\frac{(1+z)}{\xi'(1)}> 0.
		$$
		If $c_p=0$ for all even $p\geq 4,$ then $z=0$ and $\phi(0)=0$, which contradicts the above inequality. Thus, we may assume that $c_p\neq 0$ for at least one even $p\geq 4.$ From this and \eqref{cor:eq2},
		\begin{align*}
		\frac{s}{\xi'(s)}+\frac{zs}{\xi'(1)}
		\end{align*} 
		is a convex function on $(0,1)$. Since $\phi(1)=0,$ we conclude that on $(0,1)$ $\phi$ has at most one zero and therefore, $\zeta<0$ on $(0,1).$ 
		
	\end{proof}
	
	\begin{proof}[\bf Proof of Corollary \ref{cor2}]
		Our proof relies on Corollary \ref{cor1}. Note that $p,q\ge 4$ implies $\xi''(0)=0,$ so the condition \eqref{cor:eq1} is satisfied. To verify \eqref{cor:eq2}, we denote $\psi(s)=\xi'(s)/s$ and compute
		\begin{align*}
		\frac{d^2}{ds^2}\frac{1}{\psi(s)}&=\frac{1}{\psi(s)^2}\bigl(2\psi'(s)^2-\psi(s)\psi''(s)\bigr).
		\end{align*}
		Since
		\begin{align*}
		\psi(s)&=pcs^{p-2}+q(1-c)s^{q-2},
		\end{align*}
		a long computation yields
		\begin{align*}
		2\psi'(s)^2-\psi(s)\psi''(s)&=c^2(p-2)(p-2)p^2s^{2p-6}+(1-c)^2(q-2)(q-1)q^2s^{2p-6}\\
		&\quad-c(1-c)pq\bigl(3(p+q)+(p-q)^2-(2pq+4)\bigr)s^{p+q-6}.
		\end{align*}
		Here the second line of this equation is nonnegative provided the assumption \eqref{cor2:eq2} is in force. 
		
	\end{proof}
	
	\bibliography{ref}
	\bibliographystyle{plain}

\end{document}